\documentclass[a4paper, 12pt]{amsart}
\usepackage{tikz, amssymb}
\usetikzlibrary{decorations.markings}

\newtheorem{thm}{Theorem}[section]
\newtheorem{prp}[thm]{Proposition}

\newtheorem{lem}[thm]{Lemma}
\theoremstyle{definition}

\theoremstyle{remark}

\numberwithin{equation}{section}

\newcommand{\CC}{\mathbb{C}}
\newcommand{\NN}{\mathbb{N}}
\newcommand{\TT}{\mathbb{T}}
\newcommand{\ZZ}{\mathbb{Z}}
\newcommand{\RR}{\mathbb{R}}
\newcommand{\QQ}{\mathbb{Q}}

\newcommand{\Bb}{\mathcal{B}}
\newcommand{\Kk}{\mathcal{K}}
\newcommand{\Mm}{\mathcal{M}}
\newcommand{\Hh}{\mathcal{H}}

\newcommand{\Aut}{\operatorname{Aut}}

\newcommand{\supp}{\operatorname{supp}}
\newcommand{\coker}{\operatorname{coker}}
\newcommand{\diag}{\operatorname{diag}}
\newcommand{\id}{\operatorname{id}}
\newcommand{\image}{\operatorname{im}}

\newcommand{\lt}{\operatorname{lt}}
\newcommand{\MCE}{\operatorname{MCE}}
\newcommand{\lsp}{\operatorname{span}}
\newcommand{\clsp}{\overline{\lsp}}
\newcommand{\Hoho}{\Lambda^{\min}}

\allowdisplaybreaks

\title[Stably finite and AF-embeddable $k$-graph algebras]
{AF-embeddability of $2$-graph algebras and quasidiagonality of $k$-graph algebras}
\author{Lisa Orloff Clark}
\author{Astrid an Huef}
\address[L.O. Clark and A. an Huef]{Department of Mathematics and Statistics\\
University of Otago\\
PO Box 56\\
Dunedin 9054\\
New Zealand}
\email[L.O. Clark]{lclark@maths.otago.ac.nz}
\email[A. an Huef]{astrid@maths.otago.ac.nz}

\author{Aidan Sims}
\address[A. Sims]{School of Mathematics and Applied Statistics\\
University of Wollongong\\
NSW 2522\\
Australia}
\email[A. Sims]{asims@uow.edu.au}

\thanks{This research was supported by the University of Otago, the University of Wollongong, the Marsden Fund of
the Royal Society of New Zealand, and the Australian Research Council. We thank John Quigg for helpful conversations.
We thank the anonymous referee for helpful and constructive suggestions.}

\date{\today}

\begin{document}

\begin{abstract}
We characterise quasidiagonality of the $C^*$-algebra of a cofinal $k$-graph in terms of
an algebraic condition involving the coordinate matrices of the graph. This result covers
all simple $k$-graph $C^*$-algebras. In the special case of cofinal $2$-graphs we further
prove that AF-embeddability, quasidiagonality and stable finiteness of the $2$-graph
algebra are all equivalent.
\end{abstract}

\maketitle

\section{Introduction}

Finite-dimensional approximation properties for $C^*$-algebras play a very important role
in their structure theory \cite{BO}. In particular, AF-embeddability and the weaker
notions of quasidiagonality and stable finiteness play an important role in recent
advances in classification theory for simple $C^*$-algebras \cite{MatuiSato:JFA2014}.
Here we determine exactly which cofinal $k$-graphs have quasidiagonal $C^*$-algebras. We
also establish that the $C^*$-algebra of a cofinal $k$-graph is quasidiagonal if and only
if it is stably finite. When $k = 2$, we prove that these conditions are also equivalent
to AF-embeddability of the $C^*$-algebra. These results cover all simple $k$-graph
$C^*$-algebras.

Our motivation, and our key tool, is a theorem of Brown \cite[Theorem~0.2]{Brown:JFA98},
which says that if $A$ is an AF algebra and $\alpha$ is an automorphism of $A$, then
AF-embeddability, quasidiagonality and stable finiteness of the crossed product
$A\times_\alpha \ZZ$ are equivalent and are characterised by a condition on the map
$K_0(\alpha)$ in $K$-theory induced by $\alpha$. Since every AF-embeddable $C^*$-algebra
is quasidiagonal and every quasidiagonal $C^*$-algebra is stably finite, the crucial
implication of Brown's theorem says that if the image of the homomorphism $1 -
K_0(\alpha)$ contains no nontrivial elements of the positive cone of $K_0(A)$ then $A
\times_\alpha \ZZ$ is AF-embeddable. Brown describes this $K$-theoretic condition  by
saying that ``$K_0(\alpha)$ compresses no elements in $K_0(A)$.''

It is well-known that a simple graph $C^*$-algebra is AF if the graph contains no cycles,
and is purely infinite otherwise \cite{KPR}. More generally, Schafhauser
\cite{Schafhauser} has proved that AF-embeddability, quasidiagonality and stable
finiteness of $C^*(E)$ are all equivalent to the absence of a cycle with an entrance in
the graph $E$. The hard implication is that the absence of a cycle with an entrance
implies AF-embeddability, and Schafhauser proves this by direct construction. But it can
also be recovered from Brown's result using the standard realisation of a graph
$C^*$-algebra, up to stable isomorphism, as a crossed-product of the $C^*$-algebra of an
associated skew-product graph. To do so, we show that both Brown's $K$-theoretic
condition and the absence of cycles with entrances in $E$ are equivalent to the condition
that the vertex matrix $A$ of $E$ satisfies $(1 - A^t)\ZZ E^0 \cap \NN E^0 = \{0\}$ (for
details, see Lemma~\ref{lem:Halpha and matrix} and Lemma~\ref{lem:no positives}.)

Characterising pure infiniteness, stable finiteness or approximate finite dimensionality
for $k$-graph $C^*$-algebras (even assuming simplicity) has proven much more complicated
than for directed graphs (see the partial results in \cite{EvansPhD, EvansSims:JFA2012,
Sims:CJM06}). Each $k$-graph algebra is a crossed product of an AF algebra by $\ZZ^k$
\cite{KP} rather than $\ZZ$. So we cannot typically apply Brown's result to understand
when the $C^*$-algebra $C^*(\Lambda)$ of a $k$-graph $\Lambda$ is quasidiagonal.
Nevertheless, we are led to investigate the relationship between stable finiteness of
$C^*(\Lambda)$ and the group $\sum^k_{i=1} (1 - A_i^t)\ZZ\Lambda^0$ where $A_i$ denotes
the vertex matrix of the $i$th coordinate subgraph of $\Lambda$, and then obtain
quasidiagonality from recent results of Tikuisis--White--Winter \cite{TWW}.

To describe our conclusions, we first recall two key concepts. A $k$-graph $\Lambda$ is
\emph{cofinal} if it is possible to reach any vertex of $\Lambda$ from some point on any
infinite path in $\Lambda$ (see page~\pageref{pg:cofinal} for details). Since a $k$-graph
$\Lambda$ is cofinal if and only if every vertex projection in $C^*(\Lambda)$ is full,
the $C^*$-algebras of cofinal $k$-graphs include all simple $k$-graph $C^*$-algebras. A
\emph{graph trace} on a $k$-graph $\Lambda$ is a function $g$ from the set of vertices of
$\Lambda$ into $[0,\infty)$ that respects the Cuntz--Krieger relation, and is
\emph{faithful} if $g(v) \not= 0$ for every vertex $v$ (see page~\pageref{pg:graphtrace}
for details). Our main result is the following:

\begin{thm}\label{thm-main}
Let $\Lambda$ be a row-finite and cofinal $k$-graph with no sources, and coordinate
matrices $A_1, \dots, A_k$.
\begin{enumerate}
\item\label{thm-main-item1} The following are equivalent.
    \begin{enumerate}
    \item\label{it:k-QD} $C^*(\Lambda)$ is quasidiagonal.
    \item\label{it:k-SF} $C^*(\Lambda)$ is stably finite.
    \item\label{it:k-MX} $\Big(\sum_{i=1}^k\image(1-A^t_i) \Big) \cap
        \NN{\Lambda^0} = \{0\}$.
    \item\label{it:k-GT} $\Lambda$ admits a faithful graph trace.
    \end{enumerate}
\item\label{thm-main-item2} If $k=2$, then the equivalent conditions
    (\ref{it:k-QD})--(\ref{it:k-GT}) hold if and only if $C^*(\Lambda)$ can be
    embedded  in an approximately finite-dimensional $C^*$-algebra.
\end{enumerate}
\end{thm}
It follows that if cofinal $k$-graphs $\Lambda$ and $\Gamma$ have the same skeleton, then
$C^*(\Lambda)$ is stably finite if and only if $C^*(\Gamma)$ is.

We prove part~(\ref{thm-main-item1}) of this theorem in Section~\ref{sec:kgraphs}. Let
$A$ be a $C^*$-algebra with cancellation in $K_0$. Brown's proof of the implications
\emph{AF-embeddability implies quasidiagonality}, \emph{quasidiagonality implies stable
finiteness}, and \emph{stable finiteness implies that $K_0(\alpha)$ compresses no
elements on $K_0(A)$} goes through, with a suitably modified version of the last
condition, to crossed products of $A$ by arbitrary discrete groups. We apply this to the
usual realisation of $C^*(\Lambda)$ up to stable isomorphism as a crossed product of an
AF algebra by $\ZZ^k$ to prove
(\ref{it:k-QD})${\implies}$(\ref{it:k-SF})${\implies}$(\ref{it:k-MX}). Results of
\cite{PRS} show that faithful graph traces on $\Lambda$ are in bijection with faithful
semifinite traces on $C^*(\Lambda)$. To prove (\ref{it:k-GT})${\implies}$(\ref{it:k-QD}),
we combine this with Tikuisis, White and Winter's striking recent theorem, which says
that every trace on a nuclear $C^*$-algebra in the UCT class is a quasidiagonal trace. To
close the circle, we use the Separating Hyperplane Theorem from convex analysis and the
finite-intersection property in $[0,1]^{\Lambda^0}$ to deduce from the matrix
condition~(\ref{it:k-MX}) that $\Lambda$ admits a faithful graph trace.

The proof of part~(\ref{thm-main-item2}) of our main theorem occupies Sections
\ref{sec:2graphs}~and~\ref{sec:cofinal2graphs}. Most of the work is in
Section~\ref{sec:2graphs}, where we deal with the situation where one of the coordinate
graphs contains no cycles. Our results in this section do not require that $\Lambda$ is
cofinal. We show that $C^*(\Lambda)$ can be realised up to stable isomorphism as a
crossed-product of the graph $C^*$-algebra of the cycle-free coordinate graph by an
automorphism $\alpha$. Since the coordinate graph has no cycles, its $C^*$-algebra is AF,
and so Brown's theorem implies that AF-embeddability of $C^*(\Lambda)$, quasidiagonality
of $C^*(\Lambda)$ and stable finiteness of $C^*(\Lambda)$ are all equivalent to the
condition that $K_0(\alpha)$ compresses no elements of $K_0$. The bulk of the work in
this section is involved in establishing that this $K$-theoretic condition is equivalent
to condition~(\ref{it:k-MX}).

We are then left to deal with the situation where $\Lambda$ is cofinal and has cycles of
both colours, which we consider in Section~\ref{sec:cofinal2graphs}. Since every
AF-embeddable $C^*$-algebra is stably finite, it suffices to show that if $C^*(\Lambda)$
is stably finite, then it is AF-embeddable. We use stable finiteness of $C^*(\Lambda)$ to
see that no cycle in $\Lambda$ has an entrance, and then argue directly that
$C^*(\Lambda)$ is stably isomorphic to $C(\TT^2)$ and therefore AF-embeddable. In the
final section, Section~\ref{sec:examples}, we detail three examples of applications of
our results to previously-considered classes of $2$-graphs.

\section{Background}

Let $A$ be a $C^*$-algebra. We say that $A$ is \emph{AF-embeddable} if there exists an
injective homomorphism from $A$ into an AF algebra.  A projection in $A$ is
\emph{infinite} if it is Murray-von Neumann equivalent to a proper subprojection of
itself. A projection which is not infinite is called \emph{finite}. The $C^*$-algebra $A$
is \emph{infinite} if it contains an infinite projection, and is called \emph{finite} if
it admits an approximate unit of projections and all projections in $A$ are finite; if
$A\otimes\Kk$ is finite, then $A$ is \emph{stably finite} \cite[page~7]{Rordam:EMS02}.
Finally, $A$ is \emph{quasidiagonal} if there exists a faithful representation $\pi : A
\to B(\Hh)$ such that $\pi(A)$ is a set of quasidiagonal operators in the sense of
\cite[Definition~3.5]{Brown:ASPM04} (we call $\pi$ a \emph{quasidiagonal
representation}).

We will frequently need to compute in the free abelian group on generators indexed by $X$
for a countable set $X$. We denote this group by $\ZZ X$, and regard it as the group of
finitely supported integer-valued functions on $X$; so we denote the generator
corresponding to $x \in X$ by $\delta_x$. For $a \in \ZZ X$, we also write $a_x$ for
$a(x)$.

If $A$ and $B$ are stably isomorphic and have approximate identities of projections, then
$A$ is stably finite if and only if $B$ is. The next lemma is presumably well known.

\begin{lem}\label{lem:basics}
Suppose that $C^*$-algebras $A$ and $B$ are stably isomorphic.  Then $A$ is AF-embeddable
(respectively, quasidiagonal,  AF) if and only if $B$ is.
\end{lem}
\begin{proof}
Since $A$ and $B$ are stably isomorphic, $B$ is isomorphic to a corner of $A \otimes \Kk$
(and vice-versa), so we just have to check that each of the three properties passes to
stabilisations and to corners.

If $A = \overline{\bigcup A_n}$ is AF, then so is $A \otimes \Kk = \overline{\bigcup A_n
\otimes M_n(\CC)}$; and then if $\rho$ is an AF-embedding of $A$, then $\rho \otimes
1_\Kk$ is an AF-embedding of $A \otimes \Kk$. Quasidiagonality passes to stabilisations
by \cite[Corollary~15]{Hadwin}.

AF-embeddability clearly passes to subalgebras. So does quasidiagonality: if $\pi$ is
quasidiagonal representation of $A$, it restricts to a quasidiagonal representation of
each subalgebra. It is standard that corners of AF algebras are AF
\cite[Exercise~III.2]{Davidson}.
\end{proof}

A \emph{$k$-graph} is a countable category $\Lambda$ equipped with a map $d : \Lambda \to
\NN^k$, called the \emph{degree map}, that carries composition to addition and satisfies
the \emph{factorisation property}: if $d(\lambda) = m + n$ then $\lambda$ has a unique
factorisation $\lambda = \mu\nu$ such that $d(\mu) = m$ and $d(\nu) = n$. We write
$\Lambda^n := d^{-1}(n)$, and then the factorisation property implies that $\Lambda^0$ is
precisely the collection of identity morphisms in $\Lambda$. Hence the domain and
codomain maps determine functions $r,s : \Lambda \to \Lambda^0$ such that
$r(\lambda)\lambda = \lambda = \lambda s(\lambda)$ for all $\lambda \in \Lambda$. We call
these the range and source maps, and we call elements of $\Lambda^0$ \emph{vertices}, and
other elements of $\Lambda$ \emph{paths}. We use the notational convention where, for
example, $v\Lambda^n = \{\lambda \in \Lambda^n : r(\lambda) = v\}$. We say that $\Lambda$
is \emph{row-finite} if $v\Lambda^n$ is finite for all $n,v$, and that $\Lambda$ has
\emph{no sources} if $v\Lambda^n$ is nonempty for all $n,v$. For $\mu,\nu \in \Lambda$,
we write $\Hoho(\mu,\nu)$ for the set $\{(\alpha,\beta) \in s(\mu)\Lambda \times
s(\nu)\Lambda : \alpha\mu = \beta\nu \in \Lambda^{d(\mu) \vee d(\nu)}\}$.

If $\Lambda$ is a $k$-graph and $j \le k$, then the $j$th coordinate graph of $\Lambda$,
denoted $\Lambda^{\NN e_j}$, is the directed graph with vertices $\Lambda^0$, edges
$\Lambda^{e_j}$, and range and source maps inherited from $\Lambda$. We write $A_1,
\dots, A_k \in M_{\Lambda^0}(\NN)$ for the \emph{coordinate matrices} of $\Lambda$ given
by $A_i(v,w) = |v\Lambda^{e_i} w|$. In the specific case when $k = 2$, we call the first
coordinate graph the \emph{blue subgraph} and the second coordinate graph the \emph{red
subgraph}, and then call a path in the red subgraph a red path and so forth.

A red path $\lambda = \lambda_1 \lambda_2 \cdots \lambda_n \in \Lambda$ (where each
$\lambda_i \in \Lambda^{e_2}$) is called a \emph{red cycle} if $d(\lambda)\not= 0$,
$r(\lambda) = s(\lambda)$ and $s(\lambda_i) \not= s(\lambda_j)$ for $i \not= j$. It is a
\emph{red cycle with a red entrance} if there exists $i \le n$ such that
$r(\lambda_j)\Lambda^{e_2} \not= \{\lambda_j\}$; that is, if $\lambda$ is a cycle with an
entrance in the directed graph $(\Lambda^0, \Lambda^{e_2}, r, s)$.

The $C^*$-algebra $C^*(\Lambda)$ of a row-finite $k$-graph $\Lambda$ with no sources is
the universal $C^*$-algebra generated by elements $\{s_\lambda : \lambda \in \Lambda\}$
satisfying the \emph{Cuntz--Krieger relations}:
\begin{itemize}
\item[(CK1)] $\{s_v : v \in \Lambda^0\}$ is a collection of mutually orthogonal
    projections.
\item[(CK2)] $s_\lambda s_\mu = s_{\lambda\mu}$ whenever $s(\lambda) = r(\mu)$.
\item[(CK3)] $s^*_\lambda s_\lambda = s_{s(\lambda)}$ for all $\lambda \in \Lambda$.
\item[(CK4)] $s_v = \sum_{\lambda \in v\Lambda^n} s_\lambda s^*_\lambda$ for all $v
    \in \Lambda^0$ and $n \in \NN^k$.
\end{itemize}
It follows from these relations that $s_\mu^*s_\nu = \sum_{(\alpha,\beta) \in
\Hoho(\mu,\nu)} s_\alpha s^*_\beta$ for all $\mu,\nu$ (with the convention that the empty
sum is equal to zero). We often write $p_v$ rather than $s_v$ for the projection
associated to a vertex $v \in \Lambda^0$.

Every higher-rank graph $C^*$-algebra has a countable approximate identity of
projections: enumerate $\Lambda^0 = \{v_1, v_2, \dots\}$ and then put $e_n :=
\sum^n_{i=1} p_{v_n}$. It follows from \cite{BGR} that two $k$-graph $C^*$-algebras are
stably isomorphic if and only if they are Morita equivalent.

\section{Quasidiagonality of $k$-graph $C^*$-algebras}\label{sec:kgraphs}
In this section we prove part~(\ref{thm-main-item1}) of Theorem~\ref{thm-main} (see
Theorem~\ref{thm-for-main}). We first extend the ``easy" implications in Brown's theorem
\cite[Theorem~0.2]{Brown:JFA98} from $\ZZ$ to a general discrete group $G$. The other
implications in Proposition~\ref{prp:superNate} are well-known; we just record them for
ease of reference.

\begin{prp}\label{prp:superNate} Let $\alpha:G\to\Aut A$ be an action of a discrete group $G$
on  a $C^*$-algebra $A$. Let $H_\alpha$ be the subgroup of $K_0(A)$ generated by $\{(\id
-K_0(\alpha_g)) K_0(A) : g \in G\}$. Consider the following statements:
\begin{enumerate}
\item\label{it:AFe} $A \times_\alpha G$ is AF-embeddable;
\item\label{it:QD} $A \times_\alpha G$ is quasidiagonal;
\item\label{it:SF} $A \times_\alpha G$ is stably finite;
\item\label{it:K-th} $H_\alpha \cap K_0(A)^+ = \{0\}$.
\end{enumerate}
Then (\ref{it:AFe})${}\implies{}$(\ref{it:QD}) and
(\ref{it:QD})${}\implies{}$(\ref{it:SF}). If $K_0(A)$ has cancellation, then
(\ref{it:SF})${}\implies{}$(\ref{it:K-th}).
\end{prp}

We need a technical lemma to prove the implication
(\ref{it:SF})${}\implies{}$(\ref{it:K-th}) of the proposition. The proof is an adaptation
of Brown's proof of the corresponding implication in \cite[Theorem~0.2]{Brown:JFA98}.

\begin{lem}\label{lem:SNworkhorse}
Let $k\geq 1$ and  $\alpha:G\to\Aut A$ be an action of a discrete group $G$ on  a
$C^*$-algebra $A$. For $l \ge 1$, write $\alpha^l$ for  the action of $G$ on $M_l(A)$ by
entrywise application of $\alpha$. Let $H_\alpha$ be the subgroup of $K_0(A)$ generated
by
\[
    \{(\id -K_0(\alpha_g))(a) : g\in G, a \in  K_0(A)\}.
\]
Suppose that $g_1, \dots, g_k$ are elements of $G$, and that $p_i, q_i\ (1\leq i \le k)$
and $r$ are projections in $M_l(A)$ such that $\sum_i (\id -K_0(\alpha_{g_i}))([p_i] -
[q_i]) = [r]$ and $r \not= 0$. If $K_0(A)$ has cancellation, then
\[
\diag(\alpha^l_{g_1}(p_1), \dots, \alpha^l_{g_k}(p_k), q_1, \dots, q_k, 0)
\]
is an infinite projection in $M_{2k+1}(M_l(A \times_\alpha G))$.
\end{lem}
\begin{proof}
Rearranging the expression $\sum_i (\id -K_0(\alpha_{g_i}))([p_i] - [q_i]) = [r]$, we
have
\[
[\diag(p_1, \dots p_k, \alpha^l_{g_1}(q_1), \dots, \alpha^l_{g_k}(q_k), 0)]
    = [\diag(\alpha^l_{g_1}(p_1), \dots, \alpha^l_{g_k}(p_k), q_1, \dots, q_k, r)].
\]
Since $K_0(A)$ has cancellation, there is a partial isometry $V$ such that
\begin{align*}
VV^* &= \diag(p_1, \dots, p_k, \alpha^l_{g_1}(q_1), \dots, \alpha^l_{g_k}(q_k), 0),\\
V^*V &= \diag(\alpha^l_{g_1}(p_1), \dots, \alpha^l_{g_k}(p_k), q_1, \dots, q_k, r).
\end{align*}
Let $u : g \mapsto u_g$ be the universal unitary representation of $G$ in the multiplier
algebra $\Mm(A \times_\alpha G)$, and for each $i \le k$ let $U_{g_i} = \diag(u_{g_i},
\dots, u_{g_i}) \in M_l(\Mm(A \times_\alpha G))$ so that each $U_{g_i}$ implements
$\alpha^l_{g_i}$. A quick calculation shows that
\[
W := \diag(U_{g_1} p_1, \dots, U_{g_k} p_k, q_1U^*_{g_1}, \dots, q_k U^*_{g_k}, 0) \in M_l(A \times_\alpha G)
\]
satisfies
\begin{align*}
WW^* &= \diag(\alpha^l_{g_1}(p_1), \dots, \alpha^l_{g_k}(p_k), q_1, \dots, q_k, 0),\\
W^*W &= \diag(p_1, \dots p_k, \alpha^l_{g_1}(q_1), \dots, \alpha^l_{g_k}(q_k), 0) = VV^*,\\
(WV)^*(WV) &= V^*V = \diag(\alpha^l_{g_1}(p_1), \dots, \alpha^l_{g_k}(p_k), q_1, \dots, q_k, r),\\
(WV)(WV)^* &= WW^* = \diag(\alpha^l_{g_1}(p_1), \dots, \alpha^l_{g_k}(p_k), q_1, \dots, q_k, 0).
\end{align*}
Thus $WV$ is a partial isometry. Since $r\neq 0$, this shows that $V^*V$ is equivalent to
its proper subprojection $WW^*$, and hence is an infinite projection in $M_{2k+1}(M_l(A
\times_\alpha G))$.
\end{proof}

\begin{proof}[Proof of Proposition~\ref{prp:superNate}]
The implications (\ref{it:AFe})${}\implies{}$(\ref{it:QD}) and
(\ref{it:QD})${}\implies{}$(\ref{it:SF}) are special cases of the general facts that
every AF-embeddable $C^*$-algebra is quasidiagonal \cite[Propositions
7.1.9~and~7.1.10]{BO}, and every quasidiagonal $C^*$-algebra is stably finite
\cite[Proposition~7.1.15]{BO}.

For (\ref{it:SF})${}\implies{}$(\ref{it:K-th}) we adapt the proof of the corresponding
assertion in \cite[Theorem~0.2]{Brown:JFA98}, which proceeds by proving the
contrapositive statement. Assume that $K_0(A)$ is cancellative. Suppose that $H_\alpha
\cap K_0(A)^+ \not= \{0\}$. Thus there exist $l\ge 1$ and a nonzero projection $r\in
M_l(A)$ such that $[r]\in H_\alpha$. So there are finitely many elements $g_i \in G$ and
projections $p_i, q_i\in M_l(A)$ such that $\sum_i (\id -K_0(\alpha_{g_i}))([p_i] -
[q_i]) = [r]$. Now Lemma~\ref{lem:SNworkhorse} shows that there exists $L \ge 1$ such
that $M_L(A \times_\alpha G)$ contains an infinite projection. So $A \times_\alpha G$ is
not stably finite.
\end{proof}

The following presentation of the group $H_\alpha$ when $G$ is finitely generated allows
us to rephrase the $K$-theoretic condition~(\ref{it:K-th}) of
Proposition~\ref{prp:superNate} in terms of the $k$ adjacency matrices of a $k$-graph
when we apply it to prove
Theorem~\ref{thm-main}(\ref{thm-main-item1})---see~\eqref{eq:Halpha description} below.

\begin{lem}\label{it:helper}
Let $\alpha:G\to\Aut B$ be an action of a discrete group $G$ on a $C^*$-algebra $B$. Let
$H_\alpha$ be the subgroup of $K_0(B)$ generated by $\{(\id -K_0(\alpha_g)) K_0(B) : g
\in G\}$. If $g_1, \dots, g_k$ generate $G$, then $H_\alpha = \sum^k_{i=1} (\id
-K_0(\alpha_{g_i}))K_0(B)$.
\end{lem}
\begin{proof}
It helps to name the right-hand side, so we set $R := \sum^k_{i=1} (\id
-K_0(\alpha_{g_i}))K_0(B)$. Clearly $R \subseteq H_\alpha$. For the other inclusion we
first  observe that each
\[
    (\id -K_0(\alpha_{g_i^{-1}}))K_0(B) = (K_0(\alpha_{g_i}) - 1) K_0(\alpha_{g_i^{-1}})K_0(B) \subseteq R.
\]
Fix $g \in G$, and write $g = h_1 \dots h_n$ where each $h_j \in \{g_1, \dots, g_k,
g^{-1}_1, \dots, g^{-1}_k\}$. For $a \in K_0(B)$, we have
\begin{align*}
\big(\id - K_0(\alpha_g)\big)a
    &= (\id -K_0(\alpha_{h_n}))a + (\id -K_0(\alpha_{h_{n-1}}))K_0(\alpha_{h_n}(a))\\
    &\quad+(\id -K_0(\alpha_{h_{n-2}}))K_0(\alpha_{h_{n-1}h_n}(a))
        + \cdots + (\id -K_0(\alpha_{h_1}))K_0(\alpha_{h_1^{-1}g}(a)).
\end{align*}
Each term on the right-hand side belongs to $R$. Since $R$ is closed under addition, it
follows that $\big(\id - K_0(\alpha_g)\big)a \in R$. Using again that $R$ is closed under
addition, we deduce that $H_\alpha \subseteq R$.
\end{proof}

The natural question is under what circumstances some or all of the reverse implications
in Proposition~\ref{prp:superNate} also hold. We consider this question in the context of
$k$-graphs, so our first order of business is to recall how to realise a $k$-graph
$C^*$-algebra, up to stable isomorphism, as a crossed-product of an AF algebra, and
relate condition~(\ref{it:K-th}) of Proposition~\ref{prp:superNate} to the coordinate
matrices of the $k$-graph.

Recall from \cite[Definition~5.1]{KP} that if $\Lambda$ is a row-finite $k$-graph with no
sources, and $c : \Lambda \to G$ is a functor into a discrete abelian group, then the
\emph{skew-product} $k$-graph $\Lambda \times_c G$ is equal as a set to $\Lambda \times
G$, and has structure maps \[r(\lambda, g) = (r(\lambda), g), \quad s(\lambda,g) =
(s(\lambda), g + c(\lambda)),\] $(\lambda, g)(\mu, g + c(\lambda)) = (\lambda\mu, g)$
and
$d(\lambda, g) = d(\lambda)$. There is an action $\beta^c$ of $\widehat{G}$ on
$C^*(\Lambda)$ such that $\beta^c_\chi(s_\lambda) = \chi(c(\lambda)) s_\lambda$ for all
$\lambda$, and there is an isomorphism
\[
    C^*(\Lambda \times_c G) \cong C^*(\Lambda) \times_{\beta^c} \widehat{G}
\]
that carries $s_{(\lambda, g)}$ to the function $\chi \mapsto \chi(g)s_\lambda \in
C(\widehat{G}, C^*(\Lambda))$.

For $n \in \NN^k$, we write $A_n$ for the element of $M_{\Lambda^0}(\NN)$ given by
$A_n(v,w) = |v\Lambda^n w|$. So $A_n^t(v,w) = |w \Lambda^n v|$. We then have $A_m A_n =
A_{m+n}$ by the factorisation property. Thus putting $G_n := \ZZ \Lambda^0$ for each $n
\in \ZZ^k$, and defining $A^t_{m,n} : G_m \to G_n$ for $m \le n \in \ZZ^k$ by $A^t_{m,n}
:= A^t_{n-m}$, we can form the direct limit $\varinjlim_{m \in \ZZ^k}(\ZZ \Lambda^0,
A^t_m)$. We will denote by $A^{t}_{n,\infty } : G_{n} = \mathbb{Z}\Lambda^{0} \to
\varinjlim (\mathbb{Z}\Lambda^{0}, A^{t}_{m})$ the canonical homomorphism such that
$A^{t}_{n,\infty } \circ A^{t}_{m,n} = A^{t} _{m,\infty }$ for all $m \le n$. We continue
to write $A_{i}$ for $A_{e_{i}}$ for $1 \le i \le k$; to avoid confusion, we will avoid
using the pronumerals $i,j$ for elements of $\mathbb{N}^{k}$.

\begin{lem}\label{lem:Halpha and matrix}
Let $\Lambda$ be a row-finite $k$-graph with no sources. The $C^*$-algebra $C^*(\Lambda
\times_d \ZZ^k)$ is an AF algebra, and there is an order-isomorphism $K_0(C^*(\Lambda
\times_d \ZZ^k)) \cong \varinjlim_{n \in \ZZ^k}(\ZZ \Lambda^0, A^t_n)$ that carries
$[p_{(v,n)}]$ to $A_{n,\infty}\delta_v$ for all $v,n$. There is an action $\alpha : \ZZ^k
\to \Aut C^*(\Lambda \times_d \ZZ^k)$ such that $\alpha_m(s_{(\lambda,n)}) = s_{(\lambda,
m-n)}$ for all $\lambda \in \Lambda$ and $n \le m \in \NN^k$. The $C^*$-algebra
$C^*(\Lambda)$ is stably isomorphic to $C^*(\Lambda \times_d \ZZ^k) \times_\alpha \ZZ^k$.
As in Proposition~\ref{prp:superNate}, write $H_\alpha$ for the subgroup of
$K_0(C^*(\Lambda \times_d \ZZ^k))$ generated by $\{(\id - K_0(\alpha_n)) [p_{(v,m)}] : n
\in \ZZ^k, (v,m) \in (\Lambda \times_d \ZZ^k)^0\}$. Then
\[\textstyle
H_\alpha \cap K_0(C^*(\Lambda \times_d \ZZ^k))^+ = \{0\}
    \quad\text{ if and only if }\quad
\big(\sum_{i=1}^k\image(1-A^t_i) \big) \cap \NN{\Lambda^0} = \{0\}.
\]
\end{lem}
\begin{proof}
To help keep notation manageable, let $B := C^*(\Lambda \times_d \ZZ^k)$. For $m \in
\ZZ^k$, let $B_m := \clsp\{s_{(\mu, m - d(\mu))} s^*_{(\nu, m - d(\nu))} : s(\mu) =
s(\nu)\}$. Observe that the map $b : (\Lambda \times_d \ZZ^k)^0 \to \ZZ^k$ given by
$b(\mu,m) = m$ satisfies $d(\mu,m) = b(s(\mu,m)) - b(r(\mu,m))$ for all $m$; so in the
language of \cite[Lemma~4.1]{KPSv}, we have $\underline{\delta}^0 b = d$. Applying
Lemma~4.1 of \cite{KPSv}
with $A$ equal to the trivial group $\{0\}$ and $c$ the
trivial cocycle, we deduce that for each $m$ there is an isomorphism $\pi_m : B_m \to
\bigoplus_{v \in \Lambda^0} \Kk(\ell^2(\Lambda v))$ such that
\[
\pi_m(s_{(\mu, m - d(\mu))} s^*_{(\nu, m - d(\nu))}) = \theta_{\mu,\nu} \in \Kk(\ell^2(\Lambda s(\mu)))
\subseteq \bigoplus_{v \in \Lambda^0} \Kk(\ell^2(\Lambda v)).
\]
Applying \cite[Theorem~4.2]{KPSv}, we see further that for $m \le n \in \NN^k$ there is
an endomorphism $j_{m,n} : \bigoplus_{v \in \Lambda^0} \Kk(\ell^2(\Lambda v)) \to
\bigoplus_{v \in \Lambda^0} \Kk(\ell^2(\Lambda v))$ such that $j_{m,n}(\theta_{\mu,\nu})
= \sum_{\alpha \in s(\mu)\Lambda^{n-m}} \theta_{\mu\alpha,\nu\alpha}$, that $\pi_n \circ
j_{m,n} = \pi_m$, and that there is an isomorphism
\[
\pi_{\infty} : B \to \varinjlim\Big(\bigoplus_{v \in \Lambda^0} \Kk(\ell^2(\Lambda v)), j_{m,n}\Big)
\]
such that $\pi_{\infty}(s_{(\mu, m - d(\mu))} s^*_{(\nu, m - d(\nu))}) =
j_{m,\infty}(\theta_{\mu,\nu})$ for all $\mu,\nu$. Hence $B = \overline{\bigcup_n B_n}$
is AF.

The induced map $(j_{m,n})_*$ on $K_0$ satisfies
\[
(j_{m,n})_*([\theta_{v,v}]) = \sum_{\alpha \in \Lambda^m v} [\theta_{\alpha,\alpha}]
    = \sum_{w \in \Lambda^0} |v\Lambda^{n-m} v| [\theta_{w,w}].
\]
Hence, after identifying $K_0\big(\bigoplus_{v \in \Lambda^0} \Kk(\ell^2(\Lambda
v))\big)$ with $\ZZ \Lambda^0$ via $[\theta_{v,v}] \mapsto \delta_v$, we see that the map
$(j_{m,n})_*$ is implemented by $A^t_{n-m}$. We have now arrived at the desired
description of $K_0(B)$: There is an isomorphism $\rho : K_0(B) \to \varinjlim_{m \in
\ZZ^k}(\ZZ\Lambda^0, A^t_{m})$ such that
\[
\rho([s_{(\mu,m-d(\mu))} s^*_{(\mu,m-d(\mu))}])
    = \rho([s_{(s(\mu),m)}]) = A^t_{m,\infty} \delta_v
\]
for all $v \in \Lambda^0$, $\mu \in \Lambda v$ and $m \in \ZZ^k$. Moreover, this $\rho$
satisfies
\[
\rho(K_0(B)^+) = \bigcup_{m \in \NN^k} A_{m,\infty}^t(\NN\Lambda^0).
\]

Corollary~5.3 of \cite{KP} implies that
\[
    B \cong C^*(\Lambda)\times_\gamma\TT^k
\]
where $\gamma$ is the gauge action.  This isomorphism carries the inverse of the dual
action $\hat\gamma$ on $C^*(\Lambda) \times_\gamma \TT^k$ to an action $\alpha$ of
$\ZZ^k$ on $B$ such that $\alpha_m(s_{(\mu, n)})=s_{(\mu, n-m)}$ as claimed. Hence
$B\times_\alpha\ZZ^k \cong
C^*(\Lambda)\times_\gamma\TT^k\times_{\hat{\gamma}^{-1}}\ZZ^k$, which is stably
isomorphic to $C^*(\Lambda)$ by Takai duality \cite{Takai}.

To understand $H_\alpha$, we next describe the action of $\ZZ^k$ on $K_0(B)$ induced by
$\alpha$ in terms of the isomorphism $\rho$ we have just described. The action $\alpha :
\ZZ^k \to \Aut B$ satisfies
\[
\rho([\alpha_n(p_{(v,m)})])
    = \rho([p_{(v,m-n)}])
    = A^t_{m-n,\infty} \delta_v
    = A^t_{m,\infty} (A^t_n \delta_v).
\]
Since $\rho([p_{v,m}]) = A^t_{m,\infty}(\delta_v)$, we deduce that the action $\beta :
\ZZ^k \to \Aut K_0(B)$ induced by $\alpha$ is characterised by $\beta_n \circ
A^t_{m,\infty} = A^t_{m,\infty} \circ A^t_n$.

Lemma~\ref{it:helper} gives
\begin{equation}\label{eq:Halpha description}
H_\alpha = \sum^k_{i=1} (1 - K_0(\alpha_{e_i}))(K_0(B)),
\end{equation}
and hence
\[
\rho(H_\alpha) = \sum^k_{i=1} (1 - \beta_{e_i}) \varinjlim(\ZZ\Lambda^0, j_{m,n})
    = \bigcup_{m \in \ZZ^k} A^t_{m,\infty} \Big(\sum^k_{i=1} (1 - A^t_i) \ZZ\Lambda^0\Big).
\]

Now to prove the final statement of the lemma, first suppose that $H_\alpha \cap K_0(B)^+
= \{0\}$. Then
\[
\{0\} = \rho(H_\alpha) \cap \rho(K_0(B)^+)
    \supseteq A_{0,\infty}^t\Big(\Big(\sum^k_{i=1} (1 - A_i^t)\ZZ\Lambda^0\Big) \cap \NN\Lambda^0\Big).
\]
So it suffices to show that each $A^t_{0,\infty}$ is injective on $\NN\Lambda^0$. We have
\[
\ker A^t_{0,\infty} = \bigcup_{n \in \NN^k} \ker A^t_n.
\]
Since $\Lambda$ has no sources, each $A^t_n$ is a nonnegative integer matrix with no zero
columns, giving $\ker A^t_n \cap \NN\Lambda^0 = \{0\}$ as required. Hence
$\Big(\sum^k_{i=1} (1 - A_i^t)\ZZ\Lambda^0\Big) \cap \NN\Lambda^0 = \{0\}$.

Now suppose that $H_\alpha \cap K_0(B)^+ \not= \{0\}$. For each $n$, let
$\widetilde{A}_n^t$ denote the automorphism of $\varinjlim \ZZ\Lambda^0$ induced by
$A_n^t : \ZZ \Lambda^0 \to \ZZ\Lambda^0$. Using the isomorphism $\rho$ and
equation~\eqref{eq:Halpha description}, we can choose $a_1, \dots, a_k \in \varinjlim
\ZZ\Lambda^0$ such that $\sum^k_{i=1}(1 - \widetilde{A}^t_{e_i})a_i \in (\varinjlim
\ZZ\Lambda^0)^+ \setminus \{0\}$. For each $i$, we can choose $n_i \in \ZZ^k$ and
$a_{i,0} \in \ZZ \Lambda^0$ such that $a_i = A_{n_i,\infty}(a_{i,0})$. Putting $n =
\bigvee_i n_i$ and $a_{i,1} := A^t_{n_i, n}(a_{i,0})$ for each $i$, we then have $a_i =
A^t_{n,\infty} a_{i,1}$ for each $i$. Now
\[
\sum^k_{i=1}(1 - \widetilde{A}^t_{e_i}) a
    = \sum^k_{i=1}(1 - \widetilde{A}^t_{e_i}) A^t_{n,\infty}(a_{i,1})
    = A^t_{n,\infty}\Big(\sum^k_{i=1}(1 - A^t_i) a_{i,1}\Big)
\]
belongs to $(\varinjlim\ZZ\Lambda^0)^+ \setminus \{0\}$. We therefore have
$A^t_{n,\infty}\Big(\sum^k_{i=1}(1 - A^t_i) a_{i,1}\Big) = A^t_{m,\infty}(c)$ for some $m
\in \ZZ^k$ and some $c \in \NN\Lambda^0$. Again, replacing each of $m,n$ with $m\vee n$
and applying the connecting maps, we can assume that $m = n$. So
\[
\sum^k_{i=1} (1 - A^t_i) a_{i,1} - c \in \ker A^t_{n,\infty} = \bigcup_{p \ge n} \ker A^t_p,
\]
so there exists $p$ such that $A^t_p\big(\sum^k_{i=1} (1 - A^t_i) a_{i,1} - c\big) = 0$.
We have $A^t_p c \in \NN \Lambda^0$ because $A^t_p$ is a positive matrix, and $A^t_p c
\not= 0$ because $\Lambda$ has no sources. So the elements $a_{i,2} := A^t_p a_{i,1}$
satisfy
\[
\sum^k_{i=1} (1 - A^t_i) a_{i,2}
    = A^t_p\Big(\sum^k_{i=1}(1 - A^t_i)a_{i,1}\Big)
    = A^t_p c
    \in \Big(\sum_{i=1}^k\image(1-A^t_i) \Big) \cap \NN{\Lambda^0} \setminus\{0\}.\qedhere
\]
\end{proof}

Before we can prove part~(\ref{thm-main-item1}) of our main result
Theorem~\ref{thm-main}, we need to investigate how to find graph traces on $k$-graphs.
For a row-finite $k$-graph $\Lambda$ with no sources, a function
$g:\Lambda^0\to[0,\infty)$ is a \emph{graph trace\label{pg:graphtrace} on $\Lambda$} if
\[
g(v)=\sum_{\lambda\in v\Lambda^n} g(s(\lambda))
\]
for all $v\in\Lambda^0$ and $n\in\NN^k$  \cite[Definition~3.5]{PRS}. A graph trace
$g$ is \emph{faithful} if $g(v)\neq 0$ for all  $v\in\Lambda^0$.

For the proof of the following result, we use the Separating-Hyperplane Theorem. For this
we recall some terminology. As in \cite[Theorem~1.3]{Tyr}, a \emph{hyperplane} in $\RR^n$
is a subset of the form $P = \{x \in \RR^n : x\cdot \xi = t\}$ for some $t \in R$ and
$\xi \in \RR^n \setminus \{0\}$. The sets $H_1 = \{x \in \RR^n : x \cdot \xi < t\}$ and
$H_2 = \{x \in \RR^n : x \cdot (-\xi) < -t\}$ are the \emph{open half-spaces determined
by $P$}. A set $A \subseteq \RR^n$ is \emph{affine} if $\{\lambda x + (1 - \lambda) y :
\lambda \in \RR\} \subseteq A$ whenever $x,y \in A$. A convex set $C \subseteq \RR^n$ is
relatively open if it is open in the relative topology on the smallest affine subset of
$\RR^n$ that contains it. The Separating Hyperplane Theorem \cite[Theorem~11.2]{Tyr} says
that given a relatively open convex set $C$ and an closed affine set $X$ in $\RR^n$ such
that $C \cap X = \emptyset$, there is a hyperplane $P$ such that $X \subseteq P$ and $C$
is contained in one of the half-spaces determined by $P$. Putting all the terminology
together, we obtain the consequence of the theorem that we will want to apply:

\begin{thm}[Separating-Hyperplane Theorem]\label{SHT}
Suppose that $C \subseteq \RR^n$ is convex and open, that $X \subseteq \RR^n$ is affine,
and that $C \cap X = \emptyset$. Then there exist a nonzero vector $\xi \in \RR^n$ and a
real number $t$ such that $\xi \cdot x = t$ for all $x \in X$ and $t < \xi \cdot p$ for
all $p \in C$.
\end{thm}

Recall (see Definition~2.1 and Examples~1.7(ii) of \cite{KP}) that an infinite path in a
$k$-graph $\Lambda$ is a map $x : \{(m,n) \in \NN^k : m \le n\} \to \Lambda$ such that
$d(x(m,n)) = n-m$ and $x(m,n)x(n,p) = x(m,p)$ for all $m \le n \le p$. We say that
$\Lambda$ is \emph{cofinal}\label{pg:cofinal} if, for every $v \in \Lambda^0$ and every
infinite path $x$ there exists $n \in \NN^k$ such that $v \Lambda x(n,n) \not=
\emptyset$. Proposition~A.2 of \cite{LewinSims:MPCPS10} shows that $\Lambda$ is cofinal
if and only if, for all $v,w$ in $\Lambda^0$, there exists $n \in \NN^k$ such that $v
\Lambda s(\lambda) \not= \emptyset$ for all $\lambda \in w\Lambda^n$.

We say $H \subseteq \Lambda^0$ is \emph{hereditary} if $\lambda \in \Lambda$ and
$r(\lambda) \in H$ imply $s(\lambda) \in H$.

\begin{prp}\label{prp-build-graph-trace}
Let $\Lambda$ be a row-finite $k$-graph with no sources. Let $v\in\Lambda^0$, and let
$H := s(v\Lambda) \subseteq \Lambda^0$. Then $H$ is hereditary. If
\[
    \Big(\sum_{i=1}^k\image(1-A^t_i) \Big) \cap \NN{\Lambda^0} = \{0\}
\]
then there is a graph trace $g$ on $H\Lambda$ such that $g(v)=1$. If $\Lambda$ is
cofinal, then $g$ is a faithful graph trace on $H\Lambda$, and there is a unique graph
trace $\overline{g}$ on $\Lambda$ such that $\overline{g}|_H = g$.
\end{prp}
\begin{proof}
The set $H$ is clearly hereditary. Fix $n \in \NN^k$, set $V_n:= \bigcup_{m \leq n}
s(v\Lambda^m)$, and define
\begin{align*}
Y_n := \Big\{g : \Lambda^0 \to [0,1] \mid g(v) = 1 \text{ and }& g(w)=\sum_{\alpha \in w\Lambda^{e_i}}g(s(\alpha))\\
&\text{ whenever } n_i \geq 1 \text{ and } w \in V_{n-e_i}\Big\}.
\end{align*}
We claim that $Y_n \not= \emptyset$.

To see this, let
\[
X_n :=  \lsp_{\RR}\{(1-A_i^t)\delta_{w_i} : 0\leq i \leq k, n_i\geq1\text{ and } w_i \in V_{n-e_i}\}.
\]
Then each $X_n$ is a finite-dimensional subspace, and hence a closed affine subset, of
$\RR^{V_n}$.

We show that $X_n \cap (0,\infty)^{V_n} = \emptyset$. For this, we suppose that $u \in
X_n \cap (0,\infty)^{V_n}$, and derive a contradiction. Fix $x_i\in \RR^{V_{n-e_i}}$ such
that $u=\sum_{i=1}^k (1-A_i^t)x_i$. Since $V_n$ is finite, the set $(0,\infty)^{V_n}$ is
open. Since $(x_1, \dots, x_k) \mapsto \sum^k_{i=1} (1 - A_i^t) x_i$ is continuous, we
deduce that there are $y_1, \dots, y_k \in \QQ^{V_n}$ sufficiently close to the $x_i$ to
ensure that $\sum^k_{i=1} (1 - A_i^t)y_i \in (0,\infty)^{V_n}$. Fix $N \in \NN$ such that
each $N y_i \in \ZZ^{V_n}$. Since the $A_i$ are integer matrices, we obtain
\[
\sum_{i=1}^k (1-A_i^t)Ny_i
    = N\sum_{i=1}^k (1-A_i^t)y_i
    \in \Big(\sum_{i=1}^k\image(1-A^t_i) \Big) \cap \NN{\Lambda^0},
\]
contradicting our hypothesis. So $X_n \cap (0,\infty)^{V_n} = \emptyset$.

Applying Theorem~\ref{SHT} to the affine set $X_n$ and the open convex set $(0,
\infty)^{V_n}$, we obtain a nonzero vector $\xi^n \in \RR^{V_n}$ such that
\[
    \xi^n \cdot x = t \text{ and } t < \xi^n \cdot p\quad\text{ for all $x \in X_n$ and $p \in (0,\infty)^{V_n}$.}
\]
Since $X_n$ is a subspace, we have $0 \in X_n$, forcing $t = \xi^n \cdot 0 = 0$. We then
have $0 < \xi^n \cdot p$ for every $p \in (0,\infty)^{V_n}$; applying this to $p =
\delta_w$ for each $w \in V_n$ shows that $\xi^n_w > 0$ for every $w \in V_n$.

We show that
\begin{equation}\label{eqn:ckrel}
    \xi^n_v = \sum_{\lambda \in v\Lambda^m} \xi^n_{s(\lambda)} \text{ for } m \leq n.
\end{equation}
To establish~\eqref{eqn:ckrel}, we argue by induction on $|m|:=m_1+m_2+\cdots+m_k$. If
$|m|=0$ then $m=0$ and~\eqref{eqn:ckrel} holds trivially, giving a base case. Now suppose
that~\eqref{eqn:ckrel} holds for all $m\leq n$ with $|m|= N$. Suppose that $m'\in\NN^k$
satisfies $m'\leq n$ and $|m'|=N+1$. Then $m'=m+e_i$ for some $1\leq i\leq k$ and $m\leq
n$ with $|m|=N$. By the inductive hypothesis,
\[
\xi^n_v = \sum_{\lambda \in v\Lambda^{m}}\xi^n(s(\lambda)).
\]
For $\lambda \in v\Lambda^{m}$, we have $(1-A_i^t)\delta_{s(\lambda)} \in X_n$. Since
$\xi^n\in X_n^\perp$, we obtain
\begin{equation}\label{use-perp}
0 = \xi^n \cdot((1-A_i^t)\delta_{s(\lambda)})
    = \xi^n_{s(\lambda)} - \sum_{\alpha \in s(\lambda)\Lambda^{e_i}} \xi^n_{s(\alpha)}.
\end{equation}
Now
\[
\xi^n(v)
    = \sum_{\lambda \in v\Lambda^{m}} \sum_{\alpha \in s(\lambda)\Lambda^{e_i}}\xi^n_{s(\alpha)}
    = \sum_{\mu \in v\Lambda^{m'}}\xi^n_{s(\mu)}
\]
because $(\lambda,\alpha) \mapsto \lambda\alpha$ is a bijection between the terms
appearing in the sums.  This completes the inductive step, and
proves~\eqref{eqn:ckrel}.

Take $w\in V_n$, choose $\lambda' \in v\Lambda w$ with $d(\lambda') \le n$, and put
$m := d(\lambda')$.  We have
\[
\xi^n_v
    = \sum_{\lambda \in v\Lambda^m} \xi^n_{s(\lambda)}
    \geq \xi^n_w.
\]
So $\xi^n \in [0, \xi^n_v]^{V_n}$. Since $\xi^n\neq 0$, we deduce that $\xi^n(v)
> 0$. By rescaling, we may therefore assume that $\xi^n_v = 1$, and then $\xi^n_w \in
[0,1]$ for all $w \in V_n$.

Now take any $g \in [0,1]^{\Lambda^0}$ such that $g(v) = \xi^n_v$ for all $v \in V_n$
(for example, put $g(v) = \xi^n_v$ for $v \in V_n$ and $g(v) = 0$ for $v \not\in V_n$).
We show that $g \in Y_n$. For this, fix $i$ such that $n_i > 0$ and fix $w \in
V_{n-e_i}$. Choose $m\leq n-e_i$ such that $v\Lambda^m w\neq\emptyset$. Then
$(1-A_i^t)\delta_{w} \in X_n$, and the calculation~\eqref{use-perp} with $s(\lambda)$
replaced by $w$ shows that $g(w) = \xi^n_w = \sum_{\alpha \in w\Lambda^{e_i}}
\xi^n(s(\alpha)) = \sum_{\alpha \in w\Lambda^{e_i}} g(s(\alpha))$ as needed. So $g \in
Y_n$, giving $Y_n \not= \emptyset$ as claimed.

Hence the sets $Y_{(1, \dots, 1)}, Y_{(2, \dots, 2)}, \dots$ are nonempty closed subsets
of the compact space $[0,1]^{\Lambda^0}$, and since $n \le m \in \NN^k$ implies $V_n
\subseteq V_m$, we have $Y_{(1, \dots, 1)}\supseteq Y_{(2, \dots, 2)} \supseteq \dots$.
By the finite intersection property, $\bigcap^\infty_{j=1} Y_{(j, \dots, j)}$ is
nonempty; say
\[
    g\in \bigcap_{j\in\NN} Y_{(j,\dots, j)}.
\]
Then $g$ is a graph-trace on $H\Lambda$ such that $g(v) = 1$ by definition of the
$Y_n$.

Suppose that $\Lambda$ is cofinal. Then $H\Lambda$ is also cofinal. Fix $w\in H$.
Using the characterisation of cofinality from
\cite[Proposition~A.2]{LewinSims:MPCPS10}, there exists $n \in \NN^k$ such that
$w(H\Lambda)s(\lambda) \neq \emptyset$ for all $\lambda \in v(H\Lambda)^n$.  Now
\[
    1=g(v) = \sum_{\lambda \in v(H\Lambda)^n}g(s(\lambda))
\]
implies there exists $\lambda' \in v\Lambda^n$ such that $g(s(\lambda'))>0$. Since
$(wH\Lambda)s(\lambda) \not= \emptyset$, we can choose $\xi \in w(H\Lambda)s(\lambda)$.
Then
\[
    g(w) = \sum_{\mu \in w(H\Lambda)^{d(\xi)}}g(s(\mu)) > g(s(\xi))>0.
\]
Thus $g$ is a faithful graph trace on $H\Lambda$.

We must show that $g$ extends uniquely to a faithful graph trace on $\Lambda$. For $w \in
H$, set $n_w = 0 \in \NN^k$. For each $w \in \Lambda^0 \setminus H$, use the
characterisation \cite[Proposition~A.2]{LewinSims:MPCPS10} of cofinality to choose $n_w
\in \NN^k$ such that $v \Lambda s(\lambda) \neq \emptyset$ for all $\lambda \in w
\Lambda^n$. Then $s(\lambda) \in H$ for every $w \in \Lambda^0$ and $\lambda \in
w\Lambda^{n_w}$. Define $\overline{g} : \Lambda^0 \to \RR$ by $\overline{g}(w) :=
\sum_{\lambda \in w\Lambda^{n_w}} g(s(\lambda))$ for every $w \in \Lambda^0$. By
definition of $n_w$, we have $\overline{g}(w) = g(w)$ for $w \in H$, so $\overline{g}$
agrees with $g$ on $H$. Since $g$ is a faithful graph trace, we have $\overline{g}(w)
> 0$ for all $w$. If $\tilde{g}$ is any graph trace on $\Lambda$ that extends $g$, then the
graph-trace condition forces
\[
\tilde{g}(w)
    = \sum_{\lambda \in w\Lambda^{n_w}} \tilde{g}(s(\lambda))
    = \sum_{\lambda \in w\Lambda^{n_w}} g(s(\lambda))
    = \overline{g}(w)
\]
for every $w \in \Lambda^0$, so if $\overline{g}$ is a graph trace, then it is the unique
graph trace extending $g$ as claimed. So we just need to prove that $\overline{g}$ is a
graph trace. Fix $w \in \Lambda^0$ and $n \in \NN^k$. Let
\[
N := n_w \vee \Big(\bigvee_{\lambda \in w\Lambda^n} (n + n_{s(\lambda)})\Big).
\]
Since $g$ is a graph trace, we have
\[
\overline{g}(w)
    = \sum_{\mu \in w\Lambda^{n_w}} g(s(\mu))
    = \sum_{\mu \in w\Lambda^{n_w}} \Big(\sum_{\tau \in s(\mu)\Lambda^{N - n_w}} g(s(\tau))\Big)
    = \sum_{\eta \in w\Lambda^N} g(s(\eta)).
\]
Similarly,
\begin{align*}
\sum_{\lambda \in w\Lambda^n} \overline{g}(s(\lambda))
    &= \sum_{\lambda \in w\Lambda^n}  \Big(\sum_{\rho \in s(\lambda)\Lambda^{n_{s(\lambda)}}} g(s(\rho))\Big)\\
    &= \sum_{\lambda \in w\Lambda^n}  \Big(\sum_{\rho \in s(\lambda)\Lambda^{n_{s(\lambda)}}} \Big(\sum_{\zeta \in s(\rho)\Lambda^{N - (n+n_{s(\lambda)})}} g(s(\eta))\Big)\Big)\\
    &= \sum_{\eta \in w\Lambda^N} g(s(\eta)).
\end{align*}
So $\overline{g}(w) = \sum_{\lambda \in w\Lambda^n} \overline{g}(s(\lambda))$, completing
the proof.
\end{proof}

To make use of the preceding result, we first need to understand the relationship between
the existence of a faithful graph trace and quasidiagonality of the associated
$C^*$-algebra. For this, first recall \cite[Definitions II.6.7.1~and~II.6.8.1]{Blackadar}
that a semifinite trace $\tau$ on a $C^*$-algebra $A$ is a map from the set $A^+$ of
positive elements of $A$ to $[0,\infty]$ such that $\tau^{-1}([0,\infty))$ is dense in
$A^+$ and such that for all $a,b \in A^+$ and all $\lambda > 0$, we have $\tau(a+b) =
\tau(a) + \tau(b)$, $\tau(a^*a) = \tau(aa^*)$ and $\tau(\lambda a) = \lambda\tau(a)$
(with the convention $0\cdot\infty = 0$). The semifinite trace $\tau$ is \emph{faithful}
if $\tau(a^*a) > 0$ for all $a \in A\setminus\{0\}$.

\begin{lem}\label{lem:trace->QD}
Let $\Lambda$ be a cofinal row-finite $k$-graph with no sources. Suppose that there is a
faithful graph trace on $\Lambda$. Then $C^*(\Lambda)$ is quasidiagonal and carries a
faithful semifinite trace.
\end{lem}
\begin{proof}
Let $g$ be a faithful graph trace on $\Lambda$. Proposition~3.8 of \cite{PRS} implies
that there is a faithful semifinite trace $\tau$ on $C^*(\Lambda)$ such that $\tau(p_v) =
g(v)$ for all $v$. Fix $v \in \Lambda^0$. Rescale $\tau$ so that $\tau(p_v) = 1$. Then
$\tau$ restricts to a faithful trace on $p_v C^*(\Lambda) p_v$. Since $\Lambda$ is
cofinal, \cite[Proposition~3.5]{RobSi} shows that $p_v$ is full. Theorem~5.5 of \cite{KP}
shows that $C^*(\Lambda)$ is nuclear and belongs to the UCT class, so the full corner
$p_v C^*(\Lambda) p_v$ has the same properties. Since $p_v C^*(\Lambda) p_v$ is unital,
nuclear, belongs to the UCT class and has a faithful trace, \cite[Corollary~B]{TWW} shows
that $p_v C^*(\Lambda) p_v$ is quasidiagonal. Thus $C^*(\Lambda)$ is Morita equivalent,
and hence stably isomorphic \cite{BGR}, to a quasidiagonal $C^*$-algebra. By
Lemma~\ref{lem:basics}, $C^*(\Lambda)$ is quasidiagonal.
\end{proof}

We are now ready to prove our key result characterising quasidiagonality of $k$-graph
$C^*$-algebras associated to cofinal $k$-graphs.

\begin{thm}\label{thm-for-main}
Let $\Lambda$ be a row-finite  $k$-graph with no sources.
\begin{enumerate}
\item\label{prp-for-main2} If $C^*(\Lambda)$ is stably finite, then
    $\Big(\sum_{i=1}^k\image(1-A^t_i) \Big) \cap \NN{\Lambda^0} = \{0\}$.
\item\label{prp-for-main1} If $\Lambda$ is cofinal and
    $\Big(\sum_{i=1}^k\image(1-A^t_i) \Big) \cap \NN{\Lambda^0} = \{0\}$, then
    $\Lambda$ admits a faithful graph trace, and $C^*(\Lambda)$ is quasidiagonal.
\end{enumerate}
\end{thm}
\begin{proof}
\eqref{prp-for-main2} Resume the notation of Lemma~\ref{lem:Halpha and matrix}, and let
\[
\rho : \varinjlim_{n \in \ZZ^k}(\ZZ \Lambda^0, A_{m,n}) \to K_0(C^*(\Lambda \times_d \ZZ^k))
\]
be the isomorphism described there. Since $C^*(\Lambda)$ is stably finite
Lemma~\ref{lem:Halpha and matrix} then shows that $C^*(\Lambda \times_d \ZZ^k)
\times_\alpha \ZZ^k$ is also stably finite. Since $\rho$ is an isomorphism,
Proposition~\ref{prp:superNate} implies that
\[
\rho(H_\alpha) \cap \rho(K_0(C^*(\Lambda \times_d \ZZ^k)^+) = \{0\},
\]
and then the final statement of Lemma~\ref{lem:Halpha and matrix} gives
$\Big(\sum_{i=1}^k\image(1-A^t_i)\Big) \cap \NN{\Lambda^0} = \{0\}$.

\eqref{prp-for-main1} Suppose that $\Lambda$ is cofinal and that
$\Big(\sum_{i=1}^k\image(1-A^t_i)\Big) \cap \NN{\Lambda^0} = \{0\}$.
Proposition~\ref{prp-build-graph-trace} shows that $\Lambda$ carries a faithful graph
trace, and then Lemma~\ref{lem:trace->QD} shows that $C^*(\Lambda)$ is quasidiagonal.
\end{proof}

We can now prove the first part of our main result.

\begin{proof}[Proof of Theorem~\ref{thm-main}(\ref{thm-main-item1})]
Proposition~7.1.15 of \cite{BO} shows that every quasidiagonal $C^*$-algebra is stably
finite, giving (\ref{it:k-QD})${\implies}$(\ref{it:k-SF}).
Theorem~\ref{thm-for-main}(\ref{prp-for-main2}) says that
(\ref{it:k-SF})${\implies}$(\ref{it:k-MX}). Since $\Lambda$ is cofinal,
Theorem~\ref{thm-for-main}(\ref{prp-for-main1}) shows, in particular, that
(\ref{it:k-MX})${\implies}$(\ref{it:k-GT}). Finally,
(\ref{it:k-GT})${\implies}$(\ref{it:k-QD}) follows from Lemma~\ref{lem:trace->QD}.
\end{proof}

\section{Results about 1-graphs}

To apply Brown's theorem \cite[Theorem~0.2]{Brown:JFA98} in the context of $2$-graphs in
the next section, we will need an explicit description of the positive cone
$K_0(C^*(E))^+$ in the $K_0$-group of a graph $C^*$-algebra $C^*(E)$ when $E$ has no
cycles. This is folklore, but we could not locate precisely the statement that we need in
the literature.

Let $E$ be a row-finite graph with no sources, and let $A$ denote its vertex matrix. The
$K$-theory of $C^*(E)$ is calculated by applying the Pimsner--Voiculescu sequence to the
gauge action $\gamma$ of $\TT$ on $C^*(E)$ (see \cite[Chapter~7]{CBMS}). Identifying the
graph $C^*$-algebra $C^*(E)$ with the $C^*$-algebra of the associated $1$-graph, and
applying Lemma~\ref{lem:Halpha and matrix} with $k = 1$, we see that there is an
automorphism $\beta$ of the AF algebra $C^*(E \times_d \ZZ)$ such that $C^*(E)$ is stably
isomorphic to $C^*(E \times_d \ZZ) \times_\beta \ZZ$. Moreover, there is an isomorphism
$\rho : K_0(C^*(E \times_d \ZZ)) \to \varinjlim(\ZZ E^0, A^t)$ which carries $K_0(\beta)$
to the automorphism $\widetilde{A}^t$ induced by $A^t : \ZZ E^0 \to \ZZ E^0$. Since
$C^*(E\times_d\ZZ)$ is an AF algebra, the Pimsner--Voiculescu sequence collapses, giving
$K_0(C^*(E)) \cong \coker(\id-K_0(\beta))$. Hence there is an isomorphism $\phi :
K_0(C^*(E)) \to \coker (1-A^t)$ taking $[p_v]$ to $\delta_v + \image(1 - A^t)$
\cite[Theorem~7.16]{CBMS}.

\begin{lem}\label{lem:positive cone}
Let $E$ be a row-finite directed graph with no sources and no cycles. Let $A$ denote the
vertex matrix of $E$. The isomorphism $\phi : K_0(C^*(E)) \to\coker (1 - A^t)$ carries
$K_0(C^*(E))^+$ onto $\{a + \image(1 - A^t) : a \in \NN E^0\} \subseteq \coker(1 - A^t)$.
\end{lem}
\begin{proof}
The isomorphism  $\phi : K_0(C^*(E)) \to \coker (1-A^t)$ takes  $[p_v]$ to $\delta_v +
\image(1 - A^t)$. Thus  $\NN{E^0} + \image(1 - A^t)\subseteq \phi( K_0(C^*(E))^+)$.

For the other containment, fix $a \in K_0(C^*(E))^+$. Let $\{F_n\}$ be an increasing
sequence of finite subsets  of $E^1$ such that $E^1=\cup_{n=1}^\infty F_n$. Raeburn and
Szyma\'nski show in \cite[Proof of Theorem~1.5]{RSz} that there is an increasing sequence
$\{E_{F_n}\}$  of finite subgraphs of the dual graph such that $C^*(E) = \varinjlim
C^*(E_{F_n})$.  Let $\iota_n : C^*(E_{F_n}) \to C^*(E)$ be the injection. Then there is
an $n$ such that  $a = K_0(\iota_n)(b)$ for some $b$; and  $b$ is positive in
$K_0(C^*(E_{F_n}))$ because $a$ is positive.

Since $E$ has no cycles, Lemma~1.3 of \cite{RSz} implies that $E_{F_n}$ has no cycles. By
definition of the $E_{F_n}$ in \cite[Definition~1.1]{RSz},
\[E_{F_n}^0 = F_n \cup \big(r(F_n) \cap s(F_n) \cap r(E^1 \setminus F_n)\big).\]
Inspection of the proof of \cite[Lemma~1.2]{RSz} shows that the injection $\iota_n$
carries the vertex projection $q_e$ in $C^*(E_{F_n})$ to $s_e s^*_e$ for $e \in G
\subseteq E_{F_n}^0$ and carries $q_w$ to $p_w - \sum_{e \in w G} s_e s^*_e$ for $w \in
r(F_n) \cap s(F_n) \cap r(E^1 \setminus F_n)$. Since $E$ is row-finite, we then have
\begin{equation*}
\iota_n(q_w) = \sum_{e \in w E^1
\setminus F_n} s_e s^*_e.
\end{equation*}

Since $E_{F_n}$ is a finite graph and has  no cycles, there is an isomorphism
\[\bigoplus_{v \in E_{F_n}^0, vE_{F_n}^1 = \emptyset} M_{E_{F_n}^* v}(\CC) \to C^*(E_{F_n})\] that carries
the matrix unit $\delta_v \theta_{v,v}$ in the direct summand corresponding to a source
$v$ in $E_{F_n}$ to the vertex projection $p_v \in C^*(E_{F_n})$ (see, for example,
\cite[Proposition~1.18]{CBMS}). Since $b$ is positive in $K_0(C^*(E_{F_n}))$ it follows
that it is the image of a positive element of
\[K_0\Big(\bigoplus_{v \in E_{F_n}^0, vE_{F_n}^1 = \emptyset} M_{E_{F_n}^* v}(\CC)\Big) \cong
\bigoplus_{v \in E_{F_n}^0, vE_{F_n}^1 = \emptyset} \ZZ,\]
 and so has the form $\sum_{v} b_v [q_v]$ for some nonnegative integers $b_v$. Now we have
\begin{align*}
a &= K_0(\iota_n)(b)
    = \sum_{v} b_v K_0(\iota_n)([q_v]) \\
    &= \sum_{e \in G} b_e [s_e s^*_e]
        + \sum_{w \in r(F_n) \cap s(F_n) \cap r(E^1 \setminus F_n)}
            b_w \Big[\sum_{e \in w E^1 \setminus F_n} s_e s^*_e\Big] \\
    &= \sum_{e \in F_n} b_e [p_{s(e)}]
        + \sum_{w \in r(F_n) \cap s(F_n) \cap r(E^1 \setminus F_n)}
          \sum_{e \in w E^1 \setminus F_n} b_w [p_{s(e)}] \\
    &= \phi^{-1}\Big(\sum_{e \in F_n} b_e \delta_{s(e)}
        + \sum_{w \in r(F_n) \cap s(F_n) \cap r(E^1 \setminus F_n)}
          \sum_{e \in w E^1 \setminus F_n} b_w \delta_{s(e)} +\image(1 - A^t) \Big),
\end{align*}
as required.
\end{proof}

We next reconcile Theorem~\ref{thm-for-main} with Brown's theorem
\cite[Theorem~0.2]{Brown:JFA98} and with recent results of Schafhauser
\cite{Schafhauser}. Schafhauser proves that a graph $C^*$-algebra $C^*(E)$ is
AF-embeddable if and only if $E$ contains no cycle with an entrance, and that otherwise
it is not stably finite. On the other hand, Brown's theorem combined with
Lemma~\ref{lem:Halpha and matrix} shows that $C^*(E)$ is AF-embeddable if and only if the
image of $1 - A^t$ contains no nontrivial positive elements. Combining the two, we deduce
that $\image(1 - A^t)$ contains no nontrivial positive elements if and only if $E$
contains no cycle with an entrance; but it should not be necessary to go through
$C^*$-algebras to prove this combinatorial result. Since we will need one of these
implications in the next section, we present a direct proof. We emphasise that most of
the implications in the next lemma are due to Schafhauser. Our contribution is the
combinatorial proof of (\ref{it:1graphcharCycles})${\iff}$(\ref{it:1graphcharMX}).

\begin{lem}\label{lem:no positives}
Let $E$ be a row-finite directed graph with no sources. The following are equivalent:
\begin{enumerate}
\item\label{it:1graphcharAFE} $C^*(E)$ is AF-embeddable.
\item $C^*(E)$ is quasidiagonal.
\item $C^*(E)$ is stably finite.
\item $C^*(E)$ is finite.
\item\label{it:1graphcharCycles} No cycle in $E$ has an entrance.
\item\label{it:1graphcharMX} The vertex matrix $A$ of $E$ satisfies
    \[
        \image(1 - A^t) \cap \NN{E^0} = \{0\}.
    \]
\item\label{it:1graphcharK0} The automorphism $\alpha$ of $C^*(E \times_d \ZZ)$ such
    that $\alpha(s_{(e,n)}) = s_{(e, n+1)}$ satisfies
    \[
        \image(1 - K_0(\alpha)) \cap K_0(C^*(E \times_d \ZZ))^+ = \{0\}.
    \]
\end{enumerate}
\end{lem}
\begin{proof}
Schafhauser proves that (\ref{it:1graphcharAFE})--(\ref{it:1graphcharCycles}) are
equivalent in \cite[Theorem~1]{Schafhauser}. The equivalence
(\ref{it:1graphcharMX})${\iff}$(\ref{it:1graphcharK0}) follows from Lemma~\ref{lem:Halpha
and matrix} with $k=1$. So it suffices to establish
(\ref{it:1graphcharCycles})${\iff}$(\ref{it:1graphcharMX}).

For (\ref{it:1graphcharCycles})${\iff}$(\ref{it:1graphcharMX}), first suppose that
(\ref{it:1graphcharCycles}) does not hold, so $E$ has a cycle $\mu = \mu_1 \dots \mu_n$
with an entrance. Then $r(\mu_i)E^1 \not= \{\mu_i\}$ for some $i$; say $f \in r(\mu_i)E^1
\setminus \{\mu_i\}$. Let $a := \sum^n_{i=1} -\delta_{r(\mu_i)}$. Then
\begin{align*}
(1 - A^t)a
    &= \sum^n_{i=1} -\delta_{r(\mu_i)} + \sum_{e \in r(\mu_i)E^1} \delta_{s(e)}\\
    &= \sum^n_{i=1} -\delta_{r(\mu_i)}
        + \sum^n_{i=1} \delta_{s(\mu_i)}
        + \sum^n_{i=1} \sum_{e \in r(\mu_i)E^1 \setminus \{\mu_i\}} \delta_{s(e)}\\
    &\ge \delta_{s(f)}.
\end{align*}
So $(1 - A^t)a \in \big(\image(1 - A^t) \cap \NN{E^0}\big) \setminus \{0\}$.
Thus~(\ref{it:1graphcharMX}) does not hold.

Conversely, suppose that~(\ref{it:1graphcharMX}) does not hold. We must show that $E$
contains a cycle with an entrance. Choose $a \in \ZZ E^0$ such that $(1 - A^t)a \in
\NN{E^0} \setminus \{0\}$, and such that the support of $a$ is minimal in the following
sense: if $a' \in \ZZ E^0$ satisfies $(1 - A^t)a' \in \NN E^0 \setminus \{0\}$ and
$\supp(a') \subseteq \supp(a)$, then $\supp(a') = \supp(a)$.

We first claim that there is a cycle $\mu$ in $E$ such that $a_{r(\mu_i)} \not= 0$ for
every $i < |\mu|$. To see this first suppose that $a_v < 0$ for some $v \in E^0$. We have
\[
0 \le ((1 - A^t)a)_v = a_v - \sum_{e \in E^1 v} a_{r(e)},
\]
and since $a_v$ is strictly negative, it follows that there exists $e_1 \in E^1 v$ such
that $a_{r(e_1)} < 0$. Iterating this argument, we obtain edges $e_i$ such that
$s(e_{i+1}) = r(e_i)$ and $a_{r(e_i)} < 0$ for all $i$. Since $\supp(a)$ is finite, we
have $r(e_n) = s(e_m)$ for some $n \ge m$, and then $\mu = e_n \cdots e_{m}$ is the
desired cycle. Now suppose that $a_v \ge 0$ for all $v$. Since $(1 - A^t)a \not= 0$ we
have $a_v > 0$ for some $v$. Since $E$ has no sources, $v E^1$ is nonempty; choose $e_1
\in vE^1$. We have
\[
0 \le ((1 - A^t)a)_{s(e_1)}
    = a_{s(e_1)} - \sum_{f \in E^1 s(e_1)} a_{r(f)} \\
    = a_{s(e_1)} - a_v - \sum_{f \in E^1 s(e_1) \setminus \{e_1\}} a_{r(f)}.
\]
Since $a$ is nonnegative, we obtain $a_{s(e_1)} - a_v \ge 0$. Since $a_v > 0$, this
forces $a_{s(e_1)} > 0$. Again, repeating this argument gives edges $e_i$ with $s(e_i) =
r(e_{i+1})$ and $a_{s(e_i)} > 0$ for all $i$. So as above, since $\supp(a)$ is finite, we
have $r(e_n) = s(e_m)$ for some $n \ge m$ and then $\mu = e_n \dots e_m$ is the desired
cycle. This completes the proof of the claim.

Now fix a cycle $\mu = \mu_1 \dots \mu_n$ with each $a_{r(\mu_i)} \not= 0$. We prove that
$\mu$ has an entrance. To see this, we suppose to the contrary that $r(\mu_i) E^1 =
\{\mu_i\}$ for each $i$, and derive a contradiction. Since each $r(\mu_i)E^1 =
\{\mu_i\}$, the element $b := \sum^n_{i=1} \delta_{r(\mu_i)}$ belongs to $\ker(1 - A^t)$,
and since each $r(\mu_i)$ belongs to the support of $a$, the support of $b$ is contained
in $\supp(a)$. Hence
\[
a' := a - a_{r(\mu)} b
\]
satisfies $\supp(a') \subseteq \supp(a) \setminus \{r(\mu)\} \subsetneq \supp(a)$, and
\[
(1 - A^t)a' = (1 - A^t)a - a_{r(\mu)}(1 - A^t)b = (1 - A^t)a \in \NN E^0 \setminus \{0\},
\]
which contradicts minimality of the support of $a$. Thus $\mu$ is a cycle with an
entrance. This proves (\ref{it:1graphcharCycles})${\iff}$(\ref{it:1graphcharMX}).
\end{proof}

\section{2-graphs with no red cycles}\label{sec:2graphs}

In this section we show how to apply Brown's characterisation of AF-embeddability of
crossed products of AF algebras by $\ZZ$ to characterise when the $C^*$-algebra of a
2-graph $\Lambda$ with no red cycles is AF-embeddable. Symmetry gives a similar
characterisation for $2$-graphs with no blue cycles too. The main result in the section
is the following; note that, unlike in Theorem~\ref{thm-main}(\ref{thm-main-item2}), no
cofinality hypothesis is required in this result.

\begin{thm}\label{thm:1stchar2}
Let $\Lambda$ be a row-finite  $2$-graph with no sources  and let $A_i$ be the coordinate
matrices of $\Lambda$ with entries $A_i(v,w) = |v\Lambda^{e_i}w|$ for $v, w\in\Lambda^0$.
Suppose that $\Lambda$ contains no red cycles or that $\Lambda$ contains no blue cycles.
Then the following are equivalent:
\begin{enumerate}
\item\label{it:1stchar2i} $C^*(\Lambda)$ is AF-embeddable;
\item\label{it:1stchar2ii} $C^*(\Lambda)$ is quasidiagonal;
\item\label{it:1stchar2iii} $C^*(\Lambda)$ is stably finite;
\item\label{it:1stchar2iv}  $\big(\image(1-A^t_1) + \image(1-A^t_2)\big) \cap
    \NN{\Lambda^0} = \{0\}$.
\end{enumerate}
\end{thm}

To prove Theorem~\ref{thm:1stchar2}, we identify a skew-product $\Lambda \times_c \ZZ$
whose $C^*$-algebra is AF, and   an automorphism $\alpha$ of $C^*(\Lambda \times_c \ZZ)$
such that $C^*(\Lambda \times_c \ZZ)\times_\alpha\ZZ$ is stably isomorphic to
$C^*(\Lambda)$.  We then compute the range  of $K_0(\alpha)$ in $K_0(C^*(\Lambda \times_c
\ZZ))$ to apply Brown's theorem.

\begin{prp}\label{prp:skew}
Let $\Lambda$ be a row-finite $2$-graph with no sources. Define $c : \Lambda \to \NN$ by
$c(\lambda) = d(\lambda)_1$ and consider the $C^*$-algebra $C^*(\Lambda\times_c\ZZ)$ of
the skew-product graph $\Lambda\times_c\ZZ$.
\begin{enumerate}
\item\label{prp:skew1} For each $n \in \ZZ$,
\[
B_n := \clsp\{s_{(\mu,p)} s^*_{(\nu,q)} : p+c(\mu) = q+c(\nu) = n\}
\]
is a $C^*$-subalgebra of $C^*(\Lambda\times_c \ZZ)$, and we have $B_n \subseteq
B_{n+1}$ for all $n$.
\item\label{prp:skew1.5} $C^*(\Lambda \times_c \ZZ) = \overline{\bigcup_n B_n}$.
\item\label{prp:skew2} For each $n \in \ZZ$, set $P_n:=\sum_{v \in \Lambda^0}
    p_{(v,n)} \in M( C^*(\Lambda \times_c \ZZ))$. Then $P_nB_nP_n$ is a full corner
    in $B_n$ and is canonically isomorphic to  the $C^*$-algebra $C^*(\Lambda^{\NN
    e_2})$ of the directed graph $\Lambda^{\NN e_2}:=(\Lambda^0, \Lambda^{e_2}, r,
    s)$.
\item\label{prp:skew3} Suppose that $\Lambda$ has no red cycles. Then $C^*(\Lambda
    \times_c \ZZ)$ is an AF algebra.
\end{enumerate}
\end{prp}
\begin{proof}
\eqref{prp:skew1} Each $B_n$ is self-adjoint.  To see that $B_n$ is closed under
multiplication,  fix spanning elements $a:=s_{(\mu,p)}s_{(\nu,q)}^*$ and
$b:=s_{(\alpha,q')}s_{(\beta,m)}^*$ of $B_n$; it suffices to see that $ab\in B_n$. If
$r(\alpha,q') \not= r(\nu,q)$, then $ab = 0$, so we assume that $r(\alpha) = r(\nu)$ and
$q' = q$. Now $d(\nu)_1 = c(\nu) = n - q = c(\alpha) = d(\alpha)_1$, and so if
$(\sigma,\tau)\in \Hoho(\nu,\alpha)$, then  $c(\sigma)=d(\sigma)_1=0=d(\tau)_1=c(\tau)$.
We have
\begin{align*}
\Hoho((\nu, q), (\alpha, q))=\big\{\big( (\sigma, c(\nu)+q),(\tau, c(\alpha)+q) \big):(\sigma,\tau)\in\Hoho(\nu,\alpha)\big\}
\end{align*}
and hence
\begin{align*}
s_{(\mu,p)}s_{(\nu,q)}^*s_{(\alpha,q)}s_{(\beta,m)}^*
    &=s_{(\mu,p)}\Big(\sum_{(\sigma,\tau)\in\Hoho(\nu,\alpha)}  s_{(\sigma, q+c(\nu))}s_{(\tau, q+c(\alpha))}^*\Big)s_{(\beta,m)}^*\\
    &=s_{(\mu,p)}\Big(\sum_{(\sigma,\tau)\in\Hoho(\nu,\alpha)}  s_{(\sigma, p+c(\mu))}s_{(\tau, n+c(\beta))}^*\Big)s_{(\beta,m)}^*\\
    &=\sum_{(\sigma,\tau)\in\Hoho(\nu,\alpha)}  s_{(\mu\sigma,p)}s_{(\beta\tau, m)}^*.
\end{align*}
For each $(\sigma,\tau) \in \Hoho(\nu,\alpha)$, we have
$c(\mu\sigma)+p=c(\mu)+p=n=c(\beta)+m=c(\beta\tau)+m$, and hence $ab\in B_n$ as required.
It follows that $B_n$ is a $C^*$-subalgebra of $C^*(\Lambda\times_c \ZZ)$.

To see that $B_n\subseteq B_{n+1}$, let $s_{(\mu,p)}s_{(\nu,q)}^*$ in $B_n$. Then
\[
s_{(\mu,p)} s^*_{(\nu,q)}
    = \sum_{\alpha \in s(\mu)\Lambda^{e_1}} s_{(\mu,p)} s_{(\alpha, n)} s^*_{(\alpha, n)} s^*_{(\nu,q)}
    = \sum_{\alpha \in s(\mu)\Lambda^{e_1}} s_{(\mu\alpha,p)} s^*_{(\nu\alpha,q)}
\]
with $c(\mu\alpha)+p=c(\mu)+p+c(\alpha)=n+c(\alpha)=n+1=c(\nu\alpha)+q$. Thus
$s_{(\mu,p)} s^*_{(\nu,q)}\in B_{n+1}$ and it follows that $B_n\subseteq B_{n+1}$.

\eqref{prp:skew1.5} If  $s_{(\mu,p)} s^*_{(\nu,q)}$ is a nonzero spanning element of
$C^*(\Lambda\times_c\ZZ)$, then $s(\mu,p)=s(\nu,q)$ implies $c(\mu)+p=c(\nu)+q$, and
hence $s_{(\mu,p)} s^*_{(\nu,q)}\in B_{c(\mu)+p}$. Thus $C^*(\Lambda \times_c
\ZZ)=\overline{\bigcup B_n}$.

\eqref{prp:skew2} Each spanning element in $B_n$ can be written $s_{(\mu,p)}
s^*_{(\nu,q)} = s_{(\mu,p)} p_{(s(\mu), n)} s^*_{(\nu,q)} = s_{(\mu,p)} P_n
s^*_{(\nu,q)}$. Thus $B_n=   B_nP_nB_n$, and $P_nB_nP_n$ is a full corner in $B_n$.

Another calculation of the sort we have been doing shows that
\[
P_nB_nP_n=\clsp\{ s_{(\mu,n)} s^*_{(\nu,n)}:c(\mu)=c(\nu)=0 \}=\clsp\{ s_{(\mu,n)} s^*_{(\nu,n)}:\mu,\nu\in\Lambda^{\NN e_2} \}.
\]
The collection $\{p_{(v,n)}, s_{(e,n)}:v\in\Lambda^0, e\in\Lambda^{e_2}\}$ forms a
Cuntz--Krieger family for the directed graph $\Lambda^{\NN e_2}$.
It follows that there is a  homomorphism  $s_\mu \mapsto s_{(\mu,n)}$ from
$C^*(\Lambda^{\NN e_2})$ to $P_n B_n P_n$, which is surjective because the $\{p_{(v,n)},
s_{(e,n)}\}$ generate $P_n B_n P_n$.  Since each $p_{(v,n)}\neq 0$,  the gauge-invariant
uniqueness theorem \cite[Theorem~2.1]{BPRS} for directed graphs implies that  $s_\mu
\mapsto s_{(\mu,n)}$  is injective.

\eqref{prp:skew3} Since there are no red cycles in $\Lambda$, the directed graph
$\Lambda^{\NN e_2}$ has no cycles, and hence $C^*(\Lambda^{\NN e_2})$ is an AF algebra by
\cite[Theorem~2.4]{KPR}. By \eqref{prp:skew2}, $P_n B_n P_n$ and $C^*(\Lambda^{\NN e_2})$
are isomorphic, and hence $P_n B_n P_n$ is an AF algebra. Since $P_n B_n P_n$ is a full
corner of $B_n$, the two are stably isomorphic \cite{BGR}, so Lemma~\ref{lem:basics}
implies that $B_n$ is AF. Now $C^*(\Lambda \times_c \ZZ)$ is a direct limit of AF
algebras and hence is AF.
\end{proof}

Let $\Lambda$ be a row-finite  $2$-graph with no sources. Define $c : \Lambda \to \NN$ by
$c(\lambda) = d(\lambda)_1$. As in \cite[Remark~5.6]{KP}, there is an automorphism
$\alpha$ of $C^*(\Lambda \times_c \ZZ)$ such that
\begin{equation}\label{eq:alpha and morita}
\alpha(s_{(\lambda, n)}) = s_{(\lambda, n-1)}\quad\text{ for $\lambda\in\Lambda$ and $n\in\ZZ$.}
\end{equation}
The second statement of \cite[Theorem~5.7]{KP} implies that $C^*(\Lambda \times_c \ZZ)
\times_\alpha \ZZ$ is stably isomorphic to $C^*(\Lambda)$.

\begin{lem}\label{lem K of scewed graph algebra}
Let $\Lambda$ be a row-finite $2$-graph with no sources, and suppose that $\Lambda$
contains no red cycles. Define $c : \Lambda \to \NN$ by $c(\lambda) = d(\lambda)_1$.
\begin{enumerate}
\item\label{lem K of scewed graph algebra1} Let $A_1$ and $A_2$ denote the coordinate
    matrices of $\Lambda$. Then $A_1^t$ induces a homomorphism \[\widetilde{A_1^t} :
    \coker(1 - A_2^t) \to \coker(1 - A_2^t),\] and there is an isomorphism
\[
\rho : K_0(C^*(\Lambda \times_c \ZZ)) \to \varinjlim(\coker(1 - A_2^t), \widetilde{A_1^t}).
\]
Let $(\widetilde{A_1^t})_{n,\infty}$ be the canonical inclusion of the $n$th
approximating copy of $\coker(1 - A_2^t)$ in $\varinjlim(\coker(1 - A_2^t),
\widetilde{A_1^t})$. For $v\in\Lambda^0$ and $n\in\NN$  we have
\[
\rho([p_{(v,n)}])=(\widetilde{A_1^t})_{n,\infty}(\delta_v+\image(1-A_2^t)),
\]
and
\begin{equation}\label{eq:positive cone image}
\rho(K_0(C^*(\Lambda \times_c \ZZ))^+)=\bigcup_n (\widetilde{A_1^t})_{n,\infty}(\NN{\Lambda^0} + \image(1 - A_2^t)).
\end{equation}
\item\label{lem K of scewed graph algebra2} Let $\alpha$ be the automorphism of
    $C^*(\Lambda \times_c \ZZ)$ of~\eqref{eq:alpha and morita}. Then the following
    diagram commutes for each $n$:
\[
\begin{tikzpicture}[xscale=1.5]
    \node (02) at (0,2) {\small$\coker(1 - A_2^t)$};
    \node (00) at (0,0) {\small$\coker(1 - A_2^t)$};
    \node (22) at (2.3,2) {\small$\coker(1 - A_2^t)$};
    \node (20) at (2.3,0) {\small$\coker(1 - A_2^t)$};
    \node (82) at (6,2) {\small$\varinjlim(\coker((1 - A_2)^t), \widetilde{A_1^t})$};
    \node (80) at (6,0) {\small$\varinjlim(\coker((1 - A_2)^t), \widetilde{A_1^t})$};
    \draw[-stealth] (02)--(22) node[pos=0.5, below] {\small$\widetilde{A_1^t}$};
    \draw[-stealth] (00)--(20) node[pos=0.5, above] {\small$\widetilde{A_1^t}$};
    \draw[-stealth] (22)--(82) node[pos=0.5, below] {\small$(\widetilde{A_1^t})_{n+1,\infty}$};
    \draw[-stealth] (20)--(80) node[pos=0.5, above] {\small$(\widetilde{A_1^t})_{n+1, \infty}$};
    \draw[-stealth, out=340, in=200] (00) to node[pos=0.5, below] {\small$(\widetilde{A_1^t})_{n,\infty}$} (80);
    \draw[-stealth] (02)--(00) node[pos=0.5, right] {\small$\widetilde{A_1^t}$};
    \draw[-stealth] (22)--(20) node[pos=0.5, right] {\small$\widetilde{A_1^t}$};
    \draw[-stealth] (82)--(80) node[pos=0.5, right] {\small$\rho\circ K_0(\alpha)\circ\rho^{-1}$};
\end{tikzpicture}
\]
\item\label{lem K of scewed graph algebra3} Let $H_\alpha$ be the subgroup of
    $K_0(C^*(\Lambda \times_c \ZZ))$ generated by \[\{(\id -K_0(\alpha_g))
    K_0(C^*(\Lambda \times_c \ZZ)) : g \in \ZZ\}.\] Then
\begin{equation}\label{eq H alpha in action}
\begin{split}
\rho\big(H_\alpha\cap{}& K_0(C^*(\Lambda \times_c \ZZ))^+\big) \\
    &=\textstyle \bigcup_n (\widetilde{A_1^t})_{n,\infty}\Big(\big(\image(1-A_1^t) + \image(1 - A_2^t) \big)
            \cap \big(\NN{\Lambda^0}+\image(1 - A_2^t)\big)\Big).
\end{split}
\end{equation}
\end{enumerate}
\end{lem}
\begin{proof}
\eqref{lem K of scewed graph algebra1} We need to set up some notation. As in
Proposition~\ref{prp:skew}, let \[B_n := \clsp\{s_{(\mu,p)} s^*_{(\nu,q)} : p+c(\mu) =
q+c(\nu) = n\}\subseteq C^*(\Lambda \times_c \ZZ)\] and $P_n:=\sum_{v \in \Lambda^0}
p_{(v,n)} \in M( C^*(\Lambda \times_c \ZZ))$. By Proposition~\ref{prp:skew}, $B_n
\subseteq B_{n+1}$ via the Cuntz--Krieger relation, $C^*(\Lambda\times_c\ZZ)
=\overline{\bigcup B_n}$, and  $P_n B_n P_n$ is a full corner in $B_n$. The inclusion
\[i_n:P_n B_n P_n\to B_n\] induces an isomorphism $K_0(i_n):K_0(P_n B_n P_n)\to K_0(B_n)$
by \cite[Proposition~1.2]{Paschke}. Also by Proposition~\ref{prp:skew}, there is an
isomorphism
\[\phi_n:P_n B_n P_n\to C^*(\Lambda^{\NN e_2})
\]
such that $\phi_n(s_{(\mu,n)})=s_\mu$.   Let
\[\phi:K_0(C^*(\Lambda^{\NN{e_2}}))\to\coker(1-A^t_2)\] be the isomorphism of
Lemma~\ref{lem:positive cone},  which carries $[p_v]$ to $\delta_v+\image(1-A_2^t)$.

Since $\Lambda$ is a $2$-graph, the matrices $A_1$ and $A_2$ commute, and so do $A_1^t$
and $A_2^t$. So $A_1^t (1 - A_2^t)\ZZ{\Lambda^0} = (1 - A_2^t) A_1^t \ZZ{\Lambda^0}
\subseteq (1 - A_2^t)\ZZ{\Lambda^0}$, and it follows that $A_1^t$ induces a homomorphism
\[\widetilde{A_1^t} : \coker(1 - A_2^t) \to \coker(1 - A_2^t).\]

 We now consider the composition
\[\phi\circ K_0(\phi_{n+1})\circ K_0(i_{n+1})^{-1}\circ K_0( i_n):K_0(P_nB_nP_n)\to \coker(1-A^t_2).\]
Tracing a generating element $[p_{(v,n)}]\in K_0(P_nB_nP_n)^+$ through the composition we
have:
\begin{align*}
\phi\circ K_0(\phi_{n+1})&\circ K_0(i_{n+1})^{-1}\circ K_0(i_n)([p_{(v,n)}])\\
&=\phi\circ K_0(\phi_{n+1})\circ K_0(i_{n+1})^{-1}([p_{(v,n)}])\\
&=\phi\circ K_0(\phi_{n+1})\circ K_0(i_{n+1})^{-1}\Big(\sum_{e\in v\Lambda^{e_1}} [s_{(e,n)}s_{(e,n)}^*]\Big)\\
&=\phi\circ K_0(\phi_{n+1})\circ K_0(i_{n+1})^{-1}\Big(\sum_{e\in v\Lambda^{e_1}} [p_{(s(e),n+1)}]\Big)\\
&=\phi\circ K_0(\phi_{n+1})\Big(\sum_{e\in v\Lambda^{e_1}} [p_{(s(e),n+1)}]\Big)\quad\text{because $P_{n+1}\geq p_{(s(e),n+1)}$}\\
&=\phi\Big(\sum_{e\in v\Lambda^{e_1}} [p_{s(e)}]\Big)=\sum_{e\in v\Lambda^{e_1}} \delta_{s(e)}+\image(1-A_2^t).
\end{align*}
On the other hand, tracing  $[p_{(v,n)}]$ through the composition
\[\widetilde{A_1^t}\circ \phi\circ K_0(\phi_{n}):K_0(P_nB_nP_n)\to \coker(1-A^t_2)\]
we get
\begin{align*}
\widetilde{A_1^t}\circ \phi\circ K_0(\phi_{n})([p_{(v,n)}])
&=\widetilde{A_1^t}\circ \phi([p_v])
=\widetilde{A_1^t}(\delta_v+\image(1-A_2^t))\\
&=A_1^t(\delta_v)+\image(1-A_2^t)
=\sum_{e\in v\Lambda^{e_1}} \delta_{s(e)}+\image(1-A_2^t)
\end{align*}
Thus
\[
\phi\circ K_0(\phi_{n+1})\circ K_0(i_{n+1})^{-1}\circ K_0( i_n)=\widetilde{A_1^t}\circ \phi\circ K_0(\phi_{n})
\]
and precomposing both sides of this equation with $K_0(i_n)^{-1}$ shows that the next
square commutes:
\[
\begin{tikzpicture}[xscale=1.5]
    \node (02) at (0,2) {\small$K_0(B_n)$};
    \node (00) at (0,0) {\small$\coker(1 - A_2^t)$};
    \node (22) at (2.3,2) {\small$K_0(B_{n+1})$};
    \node (20) at (2.3,0) {\small$\coker(1 - A_2^t)$};
     \draw[-stealth] (02)--(22) node[pos=0.5, above] {\small$\id$};
    \draw[-stealth] (00)--(20) node[pos=0.5, above] {\small$\widetilde{A_1^t}$};
\draw[-stealth] (02)--(00) node[pos=0.5, left] {\small$\phi\circ K_0(\phi_{n})\circ K_0(i_n)^{-1}$};
    \draw[-stealth] (22)--(20) node[pos=0.5, right] {\small$\phi\circ K_0(\phi_{n+1})\circ K_0(i_{n+1})^{-1}$};
\end{tikzpicture}
\]
By the universal property of the direct limit $\overline{\bigcup K_0(B_n)}$, there is a
unique homomorphism $\rho : \overline{\bigcup K_0(B_n)} \to \varinjlim(\coker(1 - A_2^t),
\widetilde{A_1^t})$ such that $\rho|_{K_0(B_n)} =
(\widetilde{A_1^t})_{n,\infty}\circ\phi\circ K_0(\phi_{n})\circ K_0(i_n)^{-1}$. Reversing
the roles of $\overline{\bigcup K_0(B_n)}$ and $\varinjlim(\coker(1 - A_2^t),
\widetilde{A_1^t})$ implies that $\rho$ is an isomorphism.

We have  $\rho([p_{(v,n)}]=(\widetilde{A_1^t})_{n}^\infty (\delta_v + \image(1 -
A_2^t))$. So
\begin{align*}
\rho(K_0(\varinjlim B_n)^+)
    &= \bigcup \rho(K_0(B_n))^+\\
    &= \bigcup (\widetilde{A_1^t})_{n,\infty}\circ\phi\circ K_0(\phi_{n})\circ K_0(i_n)^{-1}(K_0(B_n)^+)\\
    &= \bigcup (\widetilde{A_1^t})_{n,\infty}(\NN{\Lambda^0} + \image(1 - A_2^t))
\end{align*}
by Lemma~\ref{lem:positive cone}.

\eqref{lem K of scewed graph algebra2}  The only arrow in question is the right-most down
arrow. By the universal property of $\varinjlim(\coker((1 - A_2)^t), \widetilde{A_1^t})$,
there exists a unique homomorphism \[\Upsilon: \varinjlim(\coker((1 - A_2)^t),
\widetilde{A_1^t})\to \varinjlim(\coker((1 - A_2)^t), \widetilde{A_1^t})\] such that
\[
\Upsilon\circ (\widetilde{A_1^t})_{n, \infty} =(\widetilde{A_1^t})_{n, \infty}\circ \widetilde{A_1^t}
    \quad\text{ for all $n$.}
\]
We claim that $\Upsilon=\rho\circ K_0(\alpha)\circ\rho^{-1}$. We have
\begin{gather*}
\alpha(p_{(v,n)})=\alpha\big(\sum_{e \in v\Lambda^{e_1}}s_{(e,n)}
s^*_{(e,n)}\big)=\sum_{e \in v\Lambda^{e_1}}s_{(e,n-1)} s^*_{(e,n-1)}
\end{gather*}
and hence
\[
K_0(\alpha)([p_{(v,n)}])
    = \sum_{e \in v\Lambda^{e_1}} [p_{(s(e, n-1)}]
     = \sum_{e \in v\Lambda^{e_1}} [p_{(s(e), n)}]
    = \sum_{w \in \Lambda^0} A_1^t(w,v) [p_{(w,n)}].
\]
Thus
\begin{align*}
\rho\circ K_0(\alpha)\circ\rho^{-1}\circ (\widetilde{A_1^t})_{n, \infty}(\delta_v +\image(1-A_2^t))
    &= \rho\circ K_0(\alpha)([p_{(v,n)}])\\
    &= \sum_{w\in\Lambda^0} A_i^t(w,v)\rho([p_{(w,n)}])\\
    &= \sum_{w\in\Lambda^0} A_i^t(w,v)(\widetilde{A_i^t})_{n,\infty}(\delta_w+\image(1-A_2^t))\\
    &= (\widetilde{A_i^t})_{n,\infty}\circ\widetilde{A_i^t}(\delta_v+\image(1-A_2^t)).
\end{align*}
Since the $\delta_v+\image(1-A_2^t)$ generate the positive cone of $\coker(1-A_2^t)$,
this shows that $\Upsilon=\rho\circ K_0(\alpha)\circ\rho^{-1}$ as claimed.

\eqref{lem K of scewed graph algebra3}  From~\eqref{eq:positive cone image} we have
\[\textstyle
\rho(K_0(C^*(\Lambda \times_c \ZZ))^+)
    = \bigcup_n (\widetilde{A_1^t})_{n,\infty}(\NN{\Lambda^0} + \image(1 - A_2^t)).
\]
Since $\ZZ$ is generated by $1$,  Lemma~\ref{it:helper} gives
$H_\alpha=\image(\id-K_0(\alpha))$. We have
\begin{align*}
(\widetilde{A_1^t})_{n, \infty} \big(\image(1-A_1^t) + \image(1 - A_2^t)  \big)
    &= (\widetilde{A_1^t})_{n, \infty} \big((1-\widetilde{A_1^t})\coker(1-A_2^t)\big)\\
    &= \rho\circ(\id- K_0(\alpha))\circ\rho^{-1}\circ (\widetilde{A_1^t})_{n, \infty}(\coker(1-A_2^t))\\
    &= \rho\circ (\id- K_0(\alpha))(K_0(B_n)).
\end{align*}
Thus
\begin{align*}
\rho(H_\alpha)
    &=\rho(\image(1-K_0(\alpha))) \\
    &\textstyle= \bigcup_n \rho\circ (\id- K_0(\alpha))(K_0(B_n))
     = \bigcup_n(\widetilde{A_1^t})^\infty_n \big(\image(1-A_1^t) + \image(1 - A_2^t)  \big),
\end{align*}
and the result follows.
\end{proof}

\begin{lem}\label{Halpha and the red condition}
Let $\Lambda$ be a row-finite  $2$-graph with no sources  and let $A_i$ be the coordinate
matrices of $\Lambda$.  Define $c : \Lambda \to \NN$ by $c(\lambda) = d(\lambda)_1$. Let
$\alpha$ be the automorphism of $B:=C^*(\Lambda \times_c \ZZ)$ of~\eqref{eq:alpha and
morita} and let $H_\alpha$ be the subgroup of $K_0(B)$ generated by
\[
    \{(\id -K_0(\alpha_g))(b) : g \in \ZZ, b \in K_0(B)\}.
\]
Suppose that $\Lambda$ has no red cycles. Then the following are equivalent:
\begin{enumerate}
\item\label{Halpha and the red condition1} $H_\alpha\cap K_0(B)^+=\{0\}$;
\item\label{Halpha and the red condition2} 
$\big(\image(1-A^t_1) + \image(1-A^t_2)\big) \cap \big(\NN{\Lambda^0}
    +\image(1-A^t_2)\big) = \{0_{\coker(1-A^t_2)}\}$;
\item\label{Halpha and the red condition3} $\big(\image(1-A^t_1) +
    \image(1-A^t_2)\big) \cap \NN{\Lambda^0} = \{0\}$.
\end{enumerate}
\end{lem}
\begin{proof}
We prove first that~\eqref{Halpha and the red condition1}  and~\eqref{Halpha and the red
condition2} are equivalent. Let $\rho$ be the isomorphism of Lemma~\ref{lem K of scewed
graph algebra}. By~\eqref{eq H alpha in action},
\[
\rho\big(H_\alpha\cap K_0(B)\big)
    =\bigcup_n (\widetilde{A_1^t})_{n, \infty}\Big(\big(\image(1-A_1^t) + \image(1 - A_2^t)  \big)
        \cap \big(\NN{\Lambda^0}+\image(1 - A_2^t)\big)\Big).
\]
So~\eqref{Halpha and the red condition2} immediately implies $\rho\big(H_\alpha\cap
K_0(B)^+\big)=\{0\}$, and then~\eqref{Halpha and the red condition1} follows.

Now assume~\eqref{Halpha and the red condition1}. Aiming for a contradiction, we suppose
that~\eqref{Halpha and the red condition2} fails. Then there exist $x,
y\in\ZZ{\Lambda^0}$ and $0\neq c\in \NN{\Lambda^0}$ such that $(1-A_1^t)x+(1-A_2^t)y=c$.

Since  $\rho(H_\alpha\cap K_0(B)^+)=\{0\}$, for every $n\geq 0$ we have
$(1-A_1^t)x+\image(1-A_2^t)\in \ker(\widetilde{A_1^t})_{n, \infty}$. By
\cite[Proposition~6.2.5]{RLL},
\[\textstyle
\ker(\widetilde{A^t_1})_{n, \infty} = \bigcup_{m\geq n} \ker (\widetilde{A_1^{t}}^m).
\]
So for every $n$ there exists $m\geq n$ such that  $(1-A_1^t)x+\image(1-A_2^t)\in
\ker\widetilde{A_1^{t}}^m$, that is $(A_1^t)^{m}(1-A_1^t)x\in \image(1 - A^t_2)$.

In particular, there exists $m\geq 0$ such that $(A_1^t)^{m}(1-A_1^t)x=(1-A_2^t)z$ for
some $z\in \ZZ{\Lambda^0}$. Now
\[
(A_1^t)^{m}c=(1-A_2^t)\big(z +(A_1^t)^{m}y\big)\in\NN{\Lambda^0}.
\]
Since $\Lambda$ has no sources, neither does  $\Lambda^{\NN e_1}=(\Lambda^0,
\Lambda^{e_1}, r,s)$. So $c>0$ implies $(A_1^t)^{m}c>0$. But now $0\neq (1-A_2^t)\big(z
+(A_1^t)^{m}y\big)\in\NN{\Lambda^0}$. This contradicts
(\ref{it:1graphcharCycles})${\implies}$(\ref{it:1graphcharMX}) in Lemma~\ref{lem:no
positives} applied to the directed graph $\Lambda^{\NN e_2}$, which has no cycles because
$\Lambda$ has  no red cycles. Thus~\eqref{Halpha and the red condition1}
implies~\eqref{Halpha and the red condition2}.

Next, assume~\eqref{Halpha and the red condition2}.  Fix $c\in \big(\image(1-A^t_1) +
\image(1-A^t_2)\big) \cap \NN{\Lambda^0}$, say $c=(1-A^t_1) x+(1-A^t_2) y$. Then
$c+\image(1-A^t_2)=(1-A^t_1)x+\image(1-A^t_2)=0_{\coker(1-A^t_2)}$ using~\eqref{Halpha
and the red condition2}. Now $c\in \image(1-A^t_2)\cap \NN{\Lambda^0}$.  But $\Lambda$
has no red cycles, so (\ref{it:1graphcharCycles})${\implies}$(\ref{it:1graphcharMX}) of
Lemma~\ref{lem:no positives} applied to the directed graph $\Lambda^{\NN e_2}$  gives
$\image(1-A^t_2)\cap \NN{\Lambda^0}=\{0\}$. Thus $c=0$, giving~\eqref{Halpha and the red
condition3}.

Finally, assume~\eqref{Halpha and the red condition3}. Fix  $n\in\NN{\Lambda^0}$ and
assume that
\[
n+\image(1-A^t_2) \in \big(\image(1-A^t_1) + \image(1-A^t_2)\big) \cap
    \big(\NN{\Lambda^0} +\image(1-A^t_2)\big).
\]
Then there exist $x, y\in\ZZ{\Lambda^0}$ such that $n-(1-A_1^t)x=(1-A_2^t)y$. But now $n=
(1-A_1^t)x+(1-A_2^t)y=0$ by~\eqref{Halpha and the red condition3}. Thus
$n+\image(1-A^t_2)=0_{\coker(1-A^t_2)}$. This gives~\eqref{Halpha and the red
condition2}.
\end{proof}

\begin{proof}[Proof of Theorem~\ref{thm:1stchar2}]
By symmetry, it suffices to prove the result when $\Lambda$ has no red cycles. Let
$\alpha$ be the automorphism of $B:=C^*(\Lambda\times_c\ZZ)$ described at~\eqref{eq:alpha
and morita}. Then $B\times_\alpha\ZZ$ and $C^*(\Lambda)$ are  stably isomorphic by
\cite[Theorem~5.7]{KP}. Since $\Lambda$ has no red cycles, $B$ is an AF algebra by
Proposition~\ref{prp:skew}. Thus Theorem~0.2 of \cite{Brown:JFA98} yields equivalence of
the following: $B\times_\alpha\ZZ$ is AF-embeddable, $B\times_\alpha\ZZ$ is
quasidiagonal, $B\times_\alpha\ZZ$  is stably finite, and $H_\alpha \cap K_0(B)^+ =
\{0\}$. Now Lemma~\ref{lem:basics} gives the equivalence of $\mbox{(\ref{it:1stchar2i})},
\mbox{(\ref{it:1stchar2ii})}, \mbox{(\ref{it:1stchar2iii})}$ and $H_\alpha \cap K_0(B)^+
= \{0\}$. Finally, Lemma~\ref{Halpha and the red condition} gives the equivalence of
$H_\alpha \cap K_0(B)^+ = \{0\}$ and~\eqref{it:1stchar2iv}.
\end{proof}

\section{Cofinal $2$-graphs}\label{sec:cofinal2graphs}

In this section, we consider the structure of the $C^*$-algebra of a cofinal $2$-graph
that contains both blue and red cycles, none of which have entrances. We use cofinality
to establish that $C^*(\Lambda)$ is stably isomorphic to $C(\TT^2)$, and hence
AF-embeddable. This will be the final case in our proof of
Theorem~\ref{thm-main}(\ref{thm-main-item2}).

\begin{prp}\label{prp both cycles no entries}
Let $\Lambda$ be a row-finite, cofinal $2$-graph with no sources. Suppose that $\Lambda$
contains a blue cycle with no blue entrance and a red cycle with no red entrance. Then
there exist a vertex $v$, a blue cycle $\zeta \in v\Lambda v$ and a red cycle $\xi \in
v\Lambda v$ such that $v\Lambda^\infty = \{(\zeta\xi)^\infty\}$. We have $p_v
C^*(\Lambda) p_v \cong C(\TT^2)$, the projection $p_v$ is full, and $C^*(\Lambda)$ is
stably isomorphic to $C(\TT^2)$. In particular $\Lambda$ is not aperiodic, and
$C^*(\Lambda)$ is AF-embeddable and nonsimple.
\end{prp}
\begin{proof}
Let $\mu$ be a blue cycle with no blue entrance and $\nu$ a red cycle with no red
entrance. We claim that the cycle $\mu$ has no entrance in the sense of
\cite{EvansSims:JFA2012}: that is, that for every $\eta \in r(\mu)\Lambda$ we have
$\MCE(\mu,\eta) \not= \emptyset$. To see this, fix $\eta \in r(\mu)\Lambda$. Since
$\Lambda$ has no sources, there exists $\beta \in s(\eta)\Lambda^{d(\mu)}$. Factor
$\eta\beta = \beta'\eta'$ where $d(\beta') = d(\beta) = d(\mu)$. Since
$r(\mu_i)\Lambda^{e_1} = \{\mu_i\}$ for each $i$, we have $\beta' = \mu$ and then
$\eta\beta$ is a common extension of $\mu$ and $\eta$. In particular $\MCE(\mu,\eta)
\not= \emptyset$ as required. Similarly, $\nu$ has no entrance.

Now fix an infinite path $x$ in $\Lambda$. Since $\Lambda$ is cofinal, there exists $n
\in \NN^2$ such that $r(\mu)\Lambda x(n)$ and $r(\nu)\Lambda x(n)$ are nonempty. Fix
$\lambda \in r(\mu)\Lambda x(n)$. Factorise $\lambda = \lambda_b\lambda_r$ where
$\lambda_r \in \Lambda^{\NN e_2}$ and $\lambda_b \in \Lambda^{\NN e_1}$. Since each
$r(\mu_i)\Lambda^{e_1} = \{\mu_i\}$, we have $\lambda_b = (\mu^\infty)(0, d(\lambda_b))$;
so replacing $\mu$ with $(\mu^\infty)(d(\lambda_b), d(\lambda_b) + d(\mu))$, we may
assume that $\lambda \in \Lambda^{\NN e_2}$. Since $\mu$ has no entrance, and since
$d(\lambda_r) \wedge d(\mu) = 0$, Lemma~5.6 of \cite{EvansSims:JFA2012} shows that each
$s(\lambda_r) \Lambda^{p e_1}$ is a singleton, and that there exists $p > 0$ such that
the unique element $\zeta_0$ of $s(\lambda_r)\Lambda^{p e_1}$ is a cycle. We have $kp \ge
d(\lambda_b)$ for some $k$ and then since $\lambda_b \in
s(\lambda_r)\Lambda^{|\lambda_b|e_1}$, we deduce that $\lambda_b$ is an initial segment
of $\zeta_0^\infty$. It follows that $(\zeta_0^\infty)(d(\lambda_b) , d(\lambda_b) +
pe_1)$ is a blue cycle with no entrance whose range is $x(n)$. Let $\zeta$ be the
shortest nontrivial blue cycle with no entrance such that $r(\zeta) = x(n)$. Applying the
same reasoning with the colours reversed shows that there is also a shortest red cycle
$\xi$ with no entrance such that $r(\xi) = x(n)$.

Let $v := x(n)$. We claim that $v\Lambda^\infty = \{(\zeta\xi)^\infty\}$. Since $(l\cdot
d(\zeta\xi))^\infty_{l=1}$ is a cofinal sequence in $\NN^2$, Remarks~2.2 of \cite{KP}
implies that we just need to show that each $v\Lambda^{l\cdot d(\zeta\xi)} =
\{(\zeta\xi)^l\}$. We argue by induction. Our base case is $l = 1$. Fix $\eta \in
v\Lambda^{d(\zeta\xi)}$. Express $\eta = \eta_b \eta_r$ where $d(\eta_b) = d(\zeta)$ and
$d(\eta_r) = d(\xi)$. Since $\zeta$ has no entrance, we have $r(\zeta)\Lambda^{d(\zeta)}
= \{\zeta\}$, so that $\eta_b = \zeta$. Now $r(\eta_r) = s(\zeta) = r(\xi)$, and since
$\xi$ has no entrance, we deduce that $\eta_r = \xi$. So $\eta = \eta_b\eta_r =
\zeta\xi$. Now for the inductive hypothesis, suppose that $v\Lambda^{l\cdot d(\zeta\xi)}
= \{(\zeta\xi)^l\}$. Fix $\alpha \in v\Lambda^{(l+1)\cdot d(\zeta\xi)}$. Factorise
$\alpha = \alpha'\alpha''$ with $d(\alpha') = d(\zeta\xi)$. Since $r(\alpha') = v$, the
base case gives $\alpha' = \zeta\xi$. Hence $r(\alpha'') = s(\xi) = r(\xi) = v$, and the
inductive hypothesis gives $\alpha'' = (\zeta\xi)^l$. Hence $\alpha = \alpha'\alpha'' =
(\zeta\xi)(\zeta\xi)^l = (\zeta\xi)^{l+1}$. So $v\Lambda^{(l+1)\cdot d(\zeta\xi)} =
\{(\zeta\xi)^{l+1}\}$.

Let $y = \sigma^n(x)$ so that $v\Lambda^\infty = \{y\}$. Since $v\Lambda^\infty = \{y\}$,
and since $\sigma^{d(\xi)}(y) \in v\Lambda^\infty$, we see that $\sigma^{d(\xi)}(y) = y$,
and $\sigma^{d(\zeta)}(y) = y$ by the same reasoning. We also see that $\Lambda$ is not
aperiodic since $\sigma^{d(\xi)}(z) = z$ for all $z \in v\Lambda^\infty = \{y\}$.

Let $\{S_\lambda\} \subseteq \Bb(\ell^2(\Lambda^\infty))$ be the partial isometries
defining the infinite-path space representation of $C^*(\Lambda)$, and let $\lt : \ZZ^2
\to \Bb(\ell^2(\ZZ^2))$ be the left-regular representation. By
\cite[Theorem~4.7.6]{ASPhD}, there is a faithful representation $\pi$ of $C^*(\Lambda)$
on $\ell^2(\Lambda^\infty) \otimes \ell^2(\ZZ^2)$ such that $\pi(s_\lambda) = S_\lambda
\otimes \lt_{d(\lambda)}$ for all $\lambda$. We have $p_v C^*(\Lambda) p_v = \clsp\{s_\mu
s^*_\nu : r(\mu) = r(\nu) = v\}$. Fix a spanning element $s_\mu s^*_\nu$ of $p_v
C^*(\Lambda) p_v$. Then
\[
\pi(s_\mu s^*_\nu)(h_z \otimes h_m)
    = \begin{cases}
        h_{\mu\sigma^{d(\nu)}(z)} \otimes h_{m - d(\mu) + d(\nu)} &\text{ if $z(0, d(\nu)) = \nu$}\\
        0 &\text{ otherwise.}\\
        \end{cases}
\]
Since $v\Lambda^\infty = \{y\}$, we have $v\Lambda^{d(\nu)} = \{y(0,d(\nu))\}$, and so
$z(0, d(\nu)) = \nu$ if and only if $z = y$, and then $\mu\sigma^{d(\nu)}(z) =
\mu\sigma^{d(\nu)}(y) \in v\Lambda^\infty = \{y\}$. Writing $\theta_{h_y, h_y}$ for the
rank-1 projection onto $\CC h_y \in \ell^2(\Lambda^\infty)$, we obtain
\[
\pi(s_\mu s^*_\nu) (h_z \otimes h_m)
    = \delta_{z,y} (h_y \otimes h_{m + (d(\mu) - d(\nu))})
    = \big(\theta_{h_y, h_y} \otimes \lt_{d(\mu) - d(\nu)}\big)(h_z \otimes h_m).
\]
So each $\pi(s_\mu s^*_\nu) = \theta_{h_y, h_y} \otimes \lt_{d(\mu) - d(\nu)}$. Since
$\pi$ is faithful, we deduce that $p_v C^*(\Lambda) p_v$ is isomorphic to
$C^*(\{\lt_{d(\mu) - d(\nu)} : \mu,\nu \in v\Lambda, s(\mu) = s(\nu)\})$, and hence to
the $C^*$-algebra of the subgroup $H$ of $\ZZ^2$ generated by the elements $\{d(\mu) -
d(\nu) : \mu,\nu \in v\Lambda, s(\mu) = s(\nu)\}$. In particular, $d(\xi) = d(\xi) -
d(v)$ and $d(\eta) = d(\eta) - d(v)$ both belong to $H$, and since $d(\xi) \in \NN e_1$
and $d(\eta) \in \NN e_2$, we see that the rank of $H$ is 2, and so $H \cong \ZZ^2$.
Hence $p_v C^*(\Lambda) p_v \cong C(\TT^2)$.

Since $\Lambda$ is cofinal, Proposition~3.4 of \cite{RobSi} shows that $p_v$ is full, and
then $C^*(\Lambda)$ is stably isomorphic to $p_v C^*(\Lambda) p_v$ by \cite{BGR}. It
follows immediately that $C^*(\Lambda)$ is not simple. There is a continuous surjection
of the Cantor set $2^\omega$ onto $\TT^2$ (see, for example, \cite[page~166]{Kelley}) and
hence an embedding of $C(\TT^2)$ into $C(2^\omega)$. Since $C(2^\omega)$ is AF
\cite[page~77]{Davidson}, it follows that $C(\TT^2)$ is AF-embeddable. Thus
$C^*(\Lambda)$ is AF-embeddable by Lemma~\ref{lem:basics}.
\end{proof}

The final observation that we need to complete the proof of our main theorem is that if a
$k$-graph contains a cycle with an entrance in any of its coordinate subgraphs, then its
$C^*$-algebra is not stably finite.

\begin{lem}\label{lem infinite projn}
Let $\Lambda$ be a row-finite $k$-graph with no sources, and fix $j \le k$. If there
exists a cycle with an entrance in the $j$th coordinate graph of $\Lambda$, then
$C^*(\Lambda)$ is not stably finite.
\end{lem}
\begin{proof}
This follows from the argument of \cite[Theorem~2.4]{KPR} or from
\cite[Corollary~3.8]{EvansSims:JFA2012}: if $\mu = \mu_1 \dots \mu_n$ is a cycle with an
entrance $f$ in the $j$th coordinate graph of $\Lambda$, then $S := \sum^n_{i=1}
s_{\mu_i}$ satisfies $S^*S \ge SS^* + s_fs^*_f > SS^*$, so $S^*S$ is an infinite
projection.
\end{proof}

We can now finish the proof of our main theorem.

\begin{proof}[Proof of Theorem~\ref{thm-main}(\ref{thm-main-item2})]
It suffices to show that $C^*(\Lambda)$ is AF-embeddable if and only if it is stably
finite. Every AF-embeddable $C^*$-algebra is stably finite, so we suppose that
$C^*(\Lambda)$ is stably finite. We must show that $C^*(\Lambda)$ is AF-embeddable.

First suppose either that $\Lambda$ has no red cycle or that it has no blue cycle. Then
(\ref{it:1stchar2iii})${\implies}$(\ref{it:1stchar2i}) of Theorem~\ref{thm:1stchar2}
shows that $C^*(\Lambda)$ is AF-embeddable.

Now suppose that $\Lambda$ has a red and a blue cycle. Since $C^*(\Lambda)$ is stably
finite, the contrapositive of Lemma~\ref{lem infinite projn} implies that no red cycle in
$\Lambda$ has a red entrance, and no blue cycle in $\Lambda$ has a blue entrance. Since
$\Lambda$ is cofinal and contains cycles of both colours, Proposition~\ref{prp both
cycles no entries} then implies that $C^*(\Lambda)$ is AF-embeddable.
\end{proof}

\section{Examples}\label{sec:examples}

In this section we reconcile our results with what is known about the $C^*$-algebras of
rank-2 Bratteli diagrams \cite{PRRS} and a key example from \cite{EvansSims:JFA2012}, and
indicate how to apply them to another class of concrete examples.

\subsection{Rank-2 Bratteli diagrams}

Recall from \cite{PRRS} that a row-finite $2$-graph $\Lambda$ with no sources is called a
rank-2 Bratteli diagram if there is a decomposition $\Lambda^0 = \bigcup^\infty_{n=1}
V_n$ such that: (1) each $V_n$ is finite; (2) $\Lambda^{e_1} = \bigsqcup^\infty_{n=1} V_n
\Lambda^{e_1} V_{n+1}$; (3) $\Lambda^{e_1} v \not= \emptyset$ for every $v \in \Lambda^0
\setminus V_1$; (4) $\Lambda^{e_2} = \bigsqcup^\infty_{n=1} V_n \Lambda^{e_2} V_n$; and
(5) each $v\Lambda^{e_2}$ is a singleton.

It then follows that there exist $c_n$ such that each $V_n$ decomposes as
\[\textstyle
V_n = \bigsqcup^{c_n}_{j=1} V_{n,j}
\]
such that the vertices in each $V_{n,j}$ are connected together in a red cycle with no
entrance and $\Lambda^{e_2} = \bigsqcup_{n,j} V_{n,j} \Lambda^{e_2} V_{n,j}$. Theorem~3.1
of \cite{PRRS} implies that $C^*(\Lambda)$ is an A$\TT$-algebra, and so AF-embeddable. We
show how to recover AF-embeddability from Theorem~\ref{thm:1stchar2}.

Since $\Lambda$ has no blue cycles, we must show that
\[
\big(\image(1-A^t_1) + \image(1-A^t_2)\big) \cap \NN{\Lambda^0} = \{0\}.
\]
Suppose that $a,b \in \ZZ\Lambda^0$ satisfy $(1 - A_1^t)a + (1 - A_2^t)b \ge 0$. The
matrix $A_2^t$ is block-diagonal with respect to the decomposition $\Lambda^0 =
\bigsqcup_{n,j} V_{n,j}$, and the diagonal blocks are all permutation matrices. So for
each $n,j$ we have $\sum_{v \in V_{n,j}} ((1 - A_2^t)b)_v = 0$. Thus $\sum_{v \in
V_{n,j}} ((1 - A_1^t) a)_v \ge 0$ for all $n,j$.

It suffices to prove that $(1 - A_1^t)a = 0$; for then Since $(1 - A_1^t)a + (1 - A_2^t)b
\ge 0$ forces $(1 - A_2^t)b \ge 0$, and then since $\sum_{v \in V_{n,j}} ((1 - A_2^t)b)_v
= 0$, we obtain $(1 - A_2^t)b = 0$.

For each $n,j$, let $x_{n,j} = \sum_{v \in V_{n,j}} a_v$. Lemma~4.2 of \cite{PRRS} shows
that for $v,w \in V_{n,j}$ and for $l \le c_{n+1}$, the sets $v \Lambda^{e_1} V_{n+1,l}$
and $w \Lambda^{e_1} V_{n+1, l}$ have the same cardinality $C_n(j,l)$. (Unfortunately our
notation and our convention for connectivity matrices are not compatible with
\cite{PRRS}, so our $C_n(j,l)$ is denoted $A_n(l,j)$ there.) It follows that
\begin{equation}\label{eq:bound below}
0 \le \sum_{w \in V_{n+1,l}} ((1 - A_1^t) a)_w
    = x_{n+1,l} - \sum_{j \le c_n} C_n(j,l) x_{n,j}
\end{equation}
for all $n,l$. Let $F$ be the directed graph with vertices $\{(n,j) : n \in \NN, j \le
c_n\}$ and $|(n,j) F^1 (n+1, l)| = C_n(j,l)$ for all $n,j,l$. The vertex matrix $A_F$
satisfies $A_F^t((n,j), (n+1,l)) = C_n(j,l)$, and so $(1 - A^t_F)x \in \NN F^0$. Since
$F$ has no cycles, (\ref{it:1graphcharCycles})${\implies}$(\ref{it:1graphcharMX}) gives
$(1 - A_F)x = 0$. So~\eqref{eq:bound below} gives $(1 - A_1^t) a = 0$.

\subsection{An example of Evans}

In his thesis \cite{EvansPhD}, Evans investigates conditions under which a $k$-graph
algebra is AF. He describes two particularly vexing examples that indicate the
difficulties involved in answering this question. The examples in question have common
skeleton illustrated in Figure~\ref{fig:skeleton}.
\begin{figure}[h!t]
\begin{tikzpicture}[scale=1.25]
    \foreach \x in {-3,-2,-1,0,1,2,3} {
        \foreach \y in {-1,0,1,2} {
            \node[inner sep = 1pt] (\x\y) at (\x,\y) {\tiny$(\x,\y)$};
        }
        \node at (\x,2.4) {$\vdots$};
        \node at (\x,-1.25) {$\vdots$};
    }
    \foreach \y in {-1,0,1,2} {
        \node at (-3.7,\y) {\dots};
        \node at (3.7,\y) {\dots};
    }
    \foreach \x/\xx in {-3/-2,-2/-1,-1/0,0/1,1/2,2/3} {
        \foreach \y/\yy in {0/-1,1/0,2/1} {
            \draw[-latex, red, dashed] (\x\y)--(\xx\yy);
        }
    }
    \foreach \x in {-3,-2,-1,0,1,2,3} {
        \foreach \y/\yy in {0/-1,1/0,2/1} {
        \draw[-latex, blue] (\x\y) .. controls +(0.15,-0.5) .. (\x\yy);
        \draw[-latex, red, dashed] (\x\y) .. controls +(-0.15,-0.5) .. (\x\yy);
        }
    }
    \foreach \x/\xx in {-2/-3,-1/-2,0/-1,1/0,2/1,3/2} {
        \foreach \y/\yy in {0/-1,1/0,2/1} {
            \draw[-latex, blue] (\x\y)--(\xx\yy);
        }
    }
\end{tikzpicture}
\caption{The common skeleton of two examples studied in \cite{EvansPhD}.
(For those with monochrome printers: the solid edges are blue, and the dashed edges are red)}\label{fig:skeleton}
\end{figure}
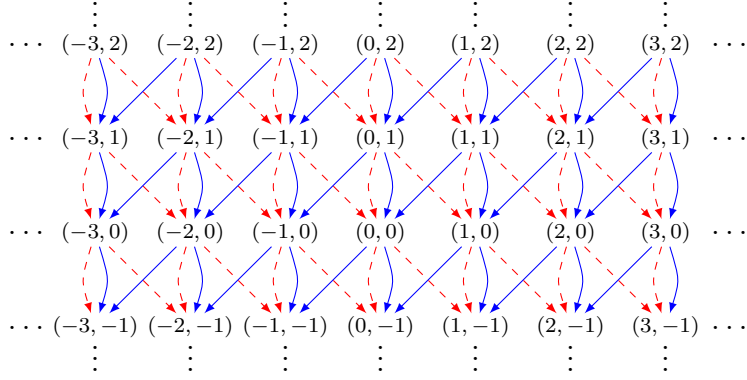

In \cite{EvansSims:JFA2012}, Evans and Sims proved that the $C^*$-algebra of one of the
two examples discussed in \cite{EvansPhD} is the UHF algebra of type $2^\infty$, and that
the other is AF-embeddable. Since the skeleton in Figure~\ref{fig:skeleton} contains no
cycles of either colour, Theorem~\ref{thm:1stchar2} applies. We show here how to see that
the $C^*$-algebra of any $2$-graph with the skeleton in Figure~\ref{fig:skeleton} is
AF-embeddable.

We must show that the coordinate matrices $A_1, A_2 \in M_{\Lambda^0}(\NN)$ for the blue
and red graphs in the skeleton satisfy $\big(\image(1-A^t_1) + \image(1-A^t_2)\big) \cap
\NN{\Lambda^0} = \{0\}$.

Fix $j \in \NN$. For all $i \in \NN$, the matrix $A_1^t$ satisfies $A_1^t \delta_{(i,j)}
= \delta_{(i,j+1)} + \delta_{(i+1, j+1)}$, and $A_2^t$ satisfies $A_2^t \delta_{(i,j)} =
\delta_{(i,j+1)} + \delta_{(i-1, j+1)}$. We deduce that for $a,b \in \ZZ\Lambda^0$, we
have
\begin{equation}\label{eq:thecondition}
\sum^\infty_{i = -\infty} \big((1 - A_1^t)a + (1 - A_2^t)b\big)_{(i,j)}
    = \sum^\infty_{i = -\infty} (a + b)_{(i,j)} - 2\sum^\infty_{i = -\infty} (a+b)_{(i, j-1)}.
\end{equation}
Now suppose that $a,b \in \ZZ\Lambda^0$ satisfy $(1 - A_1^t)a + (1 - A_2^t)b \ge 0$. Then
\[
\sum^\infty_{i = -\infty} (a + b)_{(i,j)}
     \ge 2\sum^\infty_{i = -\infty} (a+b)_{(i, j-1)}.
\]
Since $a+b$ is finitely supported, this forces $\sum^\infty_{i = -\infty} (a + b)_{(i,j)}
= 0$. Putting this back into~\eqref{eq:thecondition} forces $\sum^\infty_{i = -\infty}
\big((1 - A_1^t)a + (1 - A_2^t)b\big)_{i,j} = 0$. Since $(1 - A_1^t)a + (1 - A_2^t)b$ is
nonnegative, we deduce that it is zero.

\subsection{A class of acyclic 2-graphs}

In this section we illustrate our results with a class of examples for which the question
of AF-embeddability has not been settled by previous results. They have skeletons of the
following form (the numbers $|v_n \Lambda^{e_1} v_{n+1}|$ and $|v_m \Lambda^{e_1}
v_{m+1}|$ of blue edges connecting distinct pairs of consecutive vertices are not assumed
to be equal, and likewise for red edges). Again, solid edges are blue, and dashed edges
are red.
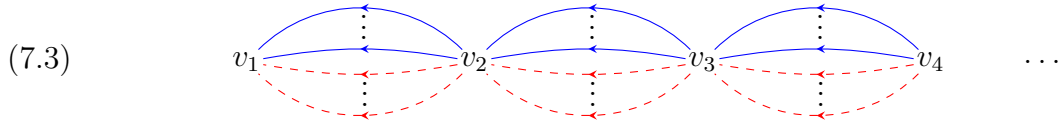
\begin{equation}\label{eq:generic2graph}
\parbox{0.8\textwidth}{\mbox{}\hfill
\begin{tikzpicture}[>=stealth, decoration={markings, mark=at position 0.5 with {\arrow{>}}}]
    \node[circle, inner sep=0pt] (v1) at (0,0) {$v_1$};
    \node[circle, inner sep=0pt] (v2) at (3,0) {$v_2$};
    \node[circle, inner sep=0pt] (v3) at (6,0) {$v_3$};
    \node[circle, inner sep=0pt] (v4) at (9,0) {$v_4$};
    \node[circle, inner sep=0pt] at (10.5,0) {$\dots$};
    \draw[blue, postaction=decorate, out=170, in=10] (v2) to node[black, above, pos=0.48, inner sep=0.3em] {$\vdots$} (v1);
    \draw[blue, postaction=decorate, out=135, in=45] (v2) to (v1);
    \draw[blue, postaction=decorate, out=170, in=10] (v3) to node[black, above, pos=0.48, inner sep=0.3em] {$\vdots$} (v2);
    \draw[blue, postaction=decorate, out=135, in=45] (v3) to (v2);
    \draw[blue, postaction=decorate, out=170, in=10] (v4) to node[black, above, pos=0.48, inner sep=0.3em] {$\vdots$} (v3);
    \draw[blue, postaction=decorate, out=135, in=45] (v4) to (v3);
    \draw[red, dashed, postaction=decorate, out=190, in=350] (v2) to (v1);
    \draw[red, dashed, postaction=decorate, out=225, in=315] (v2) to node[black, above, pos=0.48, inner sep=0.3em] {$\vdots$} (v1);
    \draw[red, dashed, postaction=decorate, out=190, in=350] (v3) to (v2);
    \draw[red, dashed, postaction=decorate, out=225, in=315] (v3) to node[black, above, pos=0.48, inner sep=0.3em] {$\vdots$} (v2);
    \draw[red, dashed, postaction=decorate, out=190, in=350] (v4) to (v3);
    \draw[red, dashed, postaction=decorate, out=225, in=315] (v4) to node[black, above, pos=0.48, inner sep=0.3em] {$\vdots$} (v3);
\end{tikzpicture}\hfill\mbox{}}
\end{equation}
In any such $2$-graph $\Lambda$, the factorisation property is completely determined by
bijections $\theta_n : v_n \Lambda^{e_1} v_{n+1} \times v_{n+1} \Lambda^{e_2} v_{n+2} \to
v_n \Lambda^{e_2} v_{n+1} \times v_{n+1} \Lambda^{e_1} v_{n+2}$: specifically,
$\alpha\beta = \beta'\alpha'$ where $\theta_n(\alpha,\beta) = (\beta',\alpha')$ (see
\cite[Section~6]{KP} or \cite{HRSW}).

Since these $2$-graphs have no cycles, Theorem~\ref{thm:1stchar2} applies to characterise
when the associated $C^*$-algebras are AF-embeddable. To apply it, we first show that the
ratios $|v_n \Lambda^{e_1}|/|v_n \Lambda^{e_2}|$ are all equal.

\begin{lem}\label{lem:unequal}
Suppose that $\Lambda$ is a $2$-graph with skeleton of the form~\eqref{eq:generic2graph}.
Then
\[
|v_n \Lambda^{e_1}|/|v_n \Lambda^{e_2}| = |v_m \Lambda^{e_1}|/|v_m \Lambda^{e_2}|\quad\text{ for all
$m,n \in \NN\setminus\{0\}$.}
\]
\end{lem}
\begin{proof}
We need only show that $|v_n \Lambda^{e_1}|/|v_n \Lambda^{e_2}| = |v_{n+1}
\Lambda^{e_1}|/|v_{n+1} \Lambda^{e_2}|$ for each $n$. The factorisation property gives
$|v_n \Lambda^{e_1}| |v_{n+1} \Lambda^{e_2}| = |v_n \Lambda^{(1,1)} v_{n+2}| = |v_n
\Lambda^{e_2}| |v_{n+1} \Lambda^{e_1}|$. Since all the quantities involved are finite and
nonzero, we can rearrange to obtain $|v_n \Lambda^{e_1}|/|v_n \Lambda^{e_2}| = |v_{n+1}
\Lambda^{e_1}|/|v_{n+1} \Lambda^{e_2}|$.
\end{proof}

We can now characterise AF-embeddability of $C^*(\Lambda)$ for any $2$-graph $\Lambda$
with skeleton of the form~\eqref{eq:generic2graph} as an immediate consequence of
Theorem~\ref{thm:1stchar2}. We could also deduce from the same result that when
$C^*(\Lambda)$ is not AF-embeddable, it is not stably finite. But with a little extra
work, we can prove that when it is not AF-embeddable, $C^*(\Lambda)$ is purely infinite.

\begin{prp}
Suppose that $\Lambda$ is a $2$-graph with skeleton of the form~\eqref{eq:generic2graph}.
Then $C^*(\Lambda)$ is AF-embeddable if $|v_1\Lambda^{e_1}| = |v_1\Lambda^{e_2}|$, and is
purely infinite otherwise.
\end{prp}
\begin{proof}
First suppose that $|v_1\Lambda^{e_1}| = |v_1\Lambda^{e_2}|$. Then $|v_n \Lambda^{e_1}| =
|v_n \Lambda^{e_2}|$ for all $n$ by Lemma~\ref{lem:unequal}. So $A_1^t = A_2^t$. Hence
$\big(\image(1-A^t_1) + \image(1-A^t_2)\big) = \image(1-A^t_1)$. Since the blue subgraph
of $\Lambda$ has no cycles, we have $\image(1 - A_1^t) \cap \NN\Lambda^0 = \{0\}$ by
(\ref{it:1graphcharCycles})${\implies}$(\ref{it:1graphcharMX}) of Lemma~\ref{lem:no
positives}. Hence $C^*(\Lambda)$ is AF-embeddable by Theorem~\ref{thm:1stchar2}.

Now suppose that $|v_1 \Lambda^{e_1}| \not= |v_1 \Lambda^{e_2}|$.  We may assume (by
reversing the roles of the colours if necessary) that $|v_1\Lambda^{e_1}| <
|v_1\Lambda^{e_2}|$. So $R := |v_1\Lambda^{e_1}|/|v_1\Lambda^{e_2}| < 1$, and
Lemma~\ref{lem:unequal} gives $|v_n \Lambda^{e_1}|/|v_n \Lambda^{e_2}| = R$ for all $n$.
We claim that $\Lambda$ is aperiodic. To see this, we suppose that $\Lambda$ is not
aperiodic, and derive a contradiction. Parts (2)~and~(3) of \cite[Theorem~4.2]{CKSS} show
that there exist $p \not= q \in \NN^2$ and an element $n \in \NN$ for which there is a
source-preserving bijection $\psi : v_n\Lambda^p \to v_n\Lambda^q$. Let $|p| := p_1 +
p_2$ be the sum of the coordinates of $p \in \NN^2$, and likewise for $q$. If $\lambda
\in v_n\Lambda^p$, then
\[
v_{n + |p|} = s(\lambda) = s(\psi(\lambda)) = v_{n + |q|},
\]
and in particular, $|p| = |q|$. Since $p \not= q$, we have $p_1 \not= q_1$, so without
loss of generality, we may assume that $p_1 > q_1$. Using the factorisation property, we
see that
\begin{align*}
|v_n \Lambda^p| &= \prod^{p_1-1}_{i=0} |v_{n+i}\Lambda^{e_1}| \prod^{|p|-1}_{i=p_1} |v_{n+i}\Lambda^{e_1}|,\quad\text{and}\\
|v_n \Lambda^q| &= \prod^{q_1-1}_{i=0} |v_{n+i}\Lambda^{e_1}| \prod^{|q|-1}_{i=q_1} |v_{n+i}\Lambda^{e_2}|
    = \prod^{q_1-1}_{i=0} |v_{n+i}\Lambda^{e_1}| \prod^{|p|-1}_{i=q_1} |v_{n+i}\Lambda^{e_2}|.
\end{align*}
Hence
\[
|v_n \Lambda^p|/|v_n \Lambda^q| = \prod^{p_1 - 1}_{i=q_1} \frac{|v_{n+i}\Lambda^{e_1}|}{|v_{n+i}\Lambda^{e_2}|} = R^{p_1 - q_1} < 1,
\]
which contradicts $|v_n\Lambda^p| = |v_n\Lambda^q|$. So $\Lambda$ is aperiodic as
claimed.

Now fix $n \in \NN$. Using that $|v_n \Lambda^{e_1}| < |v_n \Lambda^{e_2}|$, choose an
injection $\phi : v_n \Lambda^{e_1} \to v_n \Lambda^{e_2}$ and an element $\beta \in v_n
\Lambda^{e_2} \setminus \phi(v_n \Lambda^{e_1})$. The element $V := \sum_{\alpha \in v_n
\Lambda^{e_1}} s_{\phi(\alpha)} s^*_{\alpha}$ satisfies
\begin{align*}
V^*V &= \sum_{\alpha \in v_n \Lambda^{e_1}} s_\alpha s^*_\alpha
     = \sum_{\alpha \in v_n\Lambda^{e_1}} s_\alpha s^*_\alpha
     = p_{v_n}, \text{ and} \\
VV^* &= \sum_{\alpha \in v_n \Lambda^{e_1}} s_{\phi(\alpha)} s^*_{\phi(\alpha)}
     \le \sum_{\eta \in v_n\Lambda^{e_2} \setminus \beta} s_\eta s^*_\eta
     = p_{v_n} - s_\beta s^*_\beta
     < p_{v_n}.
\end{align*}
So each $p_{v_n}$ is an infinite projection.

So $\Lambda$ is cofinal and aperiodic, and all vertex projections are infinite in
$C^*(\Lambda)$. Thus \cite[Corollary~5.1]{BCS} implies that $C^*(\Lambda)$ is purely
infinite.
\end{proof}


\begin{thebibliography}{00}

\bibitem{BPRS} T. Bates, D. Pask, I. Raeburn and  W. Szyma\'nski, \emph{The
    $C^*$-algebras of row-finite graphs}, New York J. Math. \textbf{6} (2000), 307--324.

\bibitem{Blackadar} B. Blackadar, Operator algebras, Theory of
    $C\sp *$-algebras and von Neumann algebras, Operator Algebras and Non-commutative
    Geometry, III, Springer-Verlag, Berlin, 2006, xx+517.


\bibitem{BCS} J.H. Brown, L.O. Clark and A. Sierakowski, \emph{Purely infinite
    $C^*$-algebras associated to \'etale groupoids}, Ergodic Theory Dynam. Systems
    (2015) \textbf{35}, 2397--2411.

\bibitem{BGR} L.G. Brown, P. Green and M.A. Rieffel, \emph{Stable isomorphism and strong
    Morita equivalence of $C^*$-algebras}, Pacific J. Math. \textbf{71} (1977),
    349--363.

\bibitem{Brown:JFA98} N.P. Brown, \emph{A{F} embeddability of crossed products of {AF}
    algebras by the integers}, J. Funct. Anal. \textbf{160} (1998), 150--175.

\bibitem{Brown:ASPM04} N.P. Brown, \emph{On quasidiagonal {$C^*$}-algebras}, Adv. Stud.
    Pure Math., 38, Operator algebras and applications, 19--64, Math. Soc. Japan, Tokyo,
    2004.

\bibitem{BO} N.P. Brown and N. Ozawa, {$C^*$}-algebras and
    finite-dimensional approximations, American Mathematical Society, Providence, RI,
    2008, xvi+509.

\bibitem{CKSS} T.M. Carlsen, S. Kang, J. Shotwell and A. Sims, \emph{The primitive
    ideals of the {C}untz--{K}rieger algebra of a row-finite higher-rank graph with
    no sources}, J. Funct. Anal. \textbf{266} (2014), 2570--2589.

\bibitem{Davidson} K.R. Davidson, $C^*$-Algebras by Example, Amer. Math. Soc., 1996.

\bibitem{EvansPhD} D.G. Evans, On higher-rank graph {$C^*$}-algebras, PhD Thesis,
    Cardiff University, 2002.

\bibitem{EvansSims:JFA2012} D.G. Evans  and A. Sims, \emph{When is the Cuntz--Krieger
    algebra of a higher-rank graph approximately finite-dimensional?}, J. Funct. Anal.
    \textbf{263} (2012), 183--215.

\bibitem{Hadwin} D. Hadwin, \emph{Strongly quasidiagonal $C^*$-algebras. With an appendix
    by Jonathan Rosenberg},  J. Operator Theory \textbf{18} (1987),   3--18.

\bibitem{HRSW} R. Hazlewood, I. Raeburn, A. Sims and S.B.G. Webster, \emph{Remarks on
    some fundamental results about higher-rank graphs and their {$C^*$}-algebras},
    Proc. Edinb. Math. Soc. (2) \textbf{56} (2013), 575--597.

\bibitem{Kelley} J.L. Kelley, General Topology, D.~van Nostrand Company, 1968.

\bibitem{KP}  A. Kumjian and  D. Pask, \emph{Higher rank graph $C^*$-algebras}, New York
    J. Math. \textbf{6} (2000), 1--20.

\bibitem{KPR} A. Kumjian, D. Pask and I. Raeburn, \emph{Cuntz--Krieger algebras of
    directed graphs}, Pacific J. Math. \textbf{184} (1998), 161--174.

\bibitem{KPSv} A. Kumjian, D. Pask and A. Sims, \emph{On the
    {$K$}-theory of twisted higher-rank-graph {$C^*$}-algebras}, J. Math. Anal. Appl.
    \textbf{401} (2013), 104--113.

\bibitem{LewinSims:MPCPS10} P. Lewin  and A. Sims, \emph{Aperiodicity and cofinality for
    finitely aligned higher-rank graphs}, Math. Proc. Cambridge Philos. Soc. \textbf{149}
    (2010), 333--350.

\bibitem{MatuiSato:JFA2014} H. Matui and Y. Sato, \emph{Decomposition rank of
    {UHF}-absorbing {$C^*$}-algebras}, Duke Math. J. \textbf{163} (2014),
    2687--2708.

\bibitem{Paschke} W.L. Paschke, \emph{$K$-theory for actions of the circle group on
    $C^*$-algebras}, J. Operator Theory \textbf{6} (1981), 125--133.

\bibitem{PRRS} D. Pask, I. Raeburn, M. R{\o}rdam and A. Sims,
    \emph{Rank-two graphs whose {$C\sp *$}-algebras are direct limits of circle
    algebras}, J. Funct. Anal. \textbf{239} (2006), 137--178.

\bibitem{PRS} D. Pask, A. Rennie and A. Sims, \emph{The noncommutative geometry of
    $k$-graph $C^*$-algebras}, J. $K$-Theory \textbf{1} (2008), 259--304.

\bibitem{CBMS} I. Raeburn, Graph Algebras, Amer. Math. Soc., 2005.

\bibitem{RSz} I. Raeburn  and W. Szyma{\'n}ski,
    \emph{Cuntz--{K}rieger algebras of infinite graphs and matrices}, Trans. Amer. Math.
    Soc. \textbf{356} (2004), 39--59.

\bibitem{RobSi} D.I. Robertson  and A. Sims, \emph{Simplicity of {$C\sp
    \ast$}-algebras associated to higher-rank graphs}, Bull. Lond. Math. Soc. \textbf{39}
    (2007), 337--344.

\bibitem{Rordam:EMS02} M. R{\o}rdam, \emph{Classification of nuclear, simple
    $C^*$-algebras}, Encyclopaedia Math. Sci., 126, Classification of nuclear
    $C^*$-algebras. Entropy in operator algebras, 1--145, Springer, Berlin, 2002.

\bibitem{RLL} M. R{\o}rdam, F. Larsen and N.J. Laustsen, An Introduction to $K$-Theory
    for $C^*$-Algebras, London Mathematical Society Student Texts, Vol. 49, Cambridge
    University Press, 2000.

\bibitem{Schafhauser} C. Schafhauser, \emph{AF embeddings of graph algebras}, J. Operator
    Theory \textbf{74} (2015), 177--182.

\bibitem{ASPhD} A. Sims, \emph{$C^*$-algebras associated to higher-rank graphs}, PhD
    thesis, University of Newcastle,  2004.

\bibitem{Sims:CJM06} A. Sims, \emph{Gauge-invariant ideals in the {$C\sp *$}-algebras of
    finitely aligned higher-rank graphs}, Canad. J. Math. \textbf{58} (2006), 1268--1290.

\bibitem{Takai} H. Takai, \emph{On a duality for crossed products of C*-algebras}, J.
    Funct. Anal. \textbf{19} (1975), 25--39.

\bibitem{TWW}  A. Tikuisis, S. White and W. Winter, \emph{Quasidiagonality of nuclear
    C*-algebras}, preprint 2015. (arXiv:1509.08318 [math.OA]).

\bibitem{Tyr} R. Tyrrell, Convex Analysis, Princeton University Press, 1972.

\end{thebibliography}
\end{document}